\documentclass[a4paper, 11pt]{article}

\usepackage[utf8]{inputenc}
\usepackage{lmodern}
\usepackage{geometry}
\usepackage{amssymb, amsfonts, amsthm, amsmath, mathtools,bbm}
\usepackage[english]{babel}
\usepackage{hyperref}
\usepackage{tikz}
\usepackage{subcaption}
\theoremstyle{plain}
\newtheorem{theorem}{Theorem}[section]

\newtheorem{lemma}[theorem]{Lemma}

\theoremstyle{definition}

\theoremstyle{remark}
\newtheorem{remark}[theorem]{Remark}

\newcommand{\N}{\mathbb{N}}

\newcommand{\R}{\mathbb{R}}

\newcommand{\ind}[1]{\mathbbm{1}_{\left\{#1\right\}}}

\numberwithin{equation}{section}

\DeclareMathOperator{\E}{\mathbb{E}}

\renewcommand{\P}{\mathbb{P}}

\newcommand{\az}{\alpha}

\newcommand{\dd}{\mathrm{d}}

\renewcommand{\bar}[1]{\overline{#1}}
\newcommand{\egaldistr}{{\overset{(d)}{=}}}
\renewcommand{\tilde}[1]{\widetilde{#1}}

\renewcommand{\rho}{\varrho}
\renewcommand{\epsilon}{\varepsilon}

\title{Branching Brownian motion conditioned on small maximum}
\author{Xinxin Chen \quad Hui He \quad Bastien Mallein}
\date{\today}

\begin{document}

\maketitle

\begin{abstract}
We consider a standard binary branching Brownian motion on the real line. It is known that the maximal position $M_t$ among all particles alive at time $t$, shifted by $m_t = \sqrt{2} t - \frac{3}{2\sqrt{2}} \log t$ converges in law to a randomly shifted Gumbel variable. Derrida and Shi \cite{DS} conjectured the precise asymptotic behaviour of the corresponding lower deviation probability $\P(M_t \leq \sqrt{2}\alpha t)$ for $\alpha < 1$. We verify their conjecture, and describe the law of the branching Brownian motion conditioned on having a small maximum.
\end{abstract}

\section{Introduction}

We consider a one-dimensional standard binary branching Brownian motion. It is a continuous-time particle system on the real line which is constructed as follows. It starts with one individual located at the origin at time $0$ that moves according to a standard Brownian motion. After an independent exponential time of parameter $1$, the initial particle dies and gives birth to 2 children that start on the position their parent occupied at its death. These 2 children then move according to independent Brownian motions and give birth independently to their own children at rate 1. The particle system keeps evolving in this fashion for all time.

For all $t \geq 0$, we denote by $N(t)$ the collection of the individuals alive at time $t$. For any $u\in N(t)$ and $s \leq t$, let $X_u(s)$ denote the position at time $s$ of the individual $u$ or its ancestor alive at that time. The maximum of the branching Brownian motion at time $t$ is defined as
$\displaystyle
  M_t:=\max\{X_u(t): u\in N(t)\}.
$

The asymptotic behaviour of $M_t$ as $t \to \infty$ has been subjected to intense study, partly due to its link to the F-KPP reaction-diffusion equation, defined as
\begin{equation}
  \label{eqn:deffkpp}
  \partial_t u = \frac{1}{2} \Delta u -  u(1-u).
\end{equation}
Precisely, the function $(t,x) \mapsto u(x,t) = \P(M_t \leq x)$ is known to be the solution of \eqref{eqn:deffkpp} with initial condition $u(x,0) = \ind{x > 0}$.

It was proved by Bramson \cite{Bra78} that for all $z \in \R$,
\begin{equation}
  \label{eqn:bra}
  \lim_{t \to \infty} \P(M_t\leq m_t+z) = \lim_{t \to \infty} u(m_t + z, t) = w(z),
\end{equation}
where $m_t=\sqrt{2}t-\frac{3}{2\sqrt{2}}\log t$ and $w$ is the slowest travelling wave solution of the F-KPP equation, which is solution of the differential equation
\[
  \frac{1}{2} w'' + \sqrt{2} w' + w - w^2 = 0.
\]
Observe that $(t,x) \mapsto w(x-\sqrt{2}t)$ is a solution to \eqref{eqn:deffkpp}. Lalley and Sellke \cite{LS87} later provided the following representation for $w$ as
\begin{equation}
  \label{eqn:ls}
  w(z):=\E[e^{-C_0e^{-\sqrt{2}z}D_\infty}],
\end{equation}
where $C_0>0$ is a constant and $D_\infty$ is an a.s.\@ positive random variable, defined as the almost sure limit of the so-called derivative martingale, defined for all $t \geq 0$ by
\[
  D_t:=\sum_{u\in N(t)}(\sqrt{2}t-X_u(t))e^{\sqrt{2}X_u(t)-2t}.
\]

The upper large deviations of the branching Brownian motion (i.e.\@ estimating the asymptotic decay of $\P(M_t \geq \sqrt{2}\alpha t)$ for $\alpha \geq 1$) were first investigated by Chauvin and Rouault \cite{CR88,CR90}, who obtained tight estimates of $\P(M_t \geq \sqrt{2}\alpha t)$ for $\alpha \geq 1$. It is now known (cf.\@ e.g.\@ \cite[Lemma 9.7]{Bo17}) that
\begin{equation}
  \P(M_t \geq \sqrt{2}\alpha t) \sim
  \begin{cases}
    \frac{\Upsilon(\alpha)}{\sqrt{4\pi\alpha}} t^{-1/2}e^{- (\alpha^2-1)t}, &\text{if }\alpha >  1,\\[9pt]
    \frac{3C_0}{2\sqrt{2}} t^{-3/2} \log t, &\text{if }\alpha = 1,
  \end{cases}
  \text{ as } t \to \infty.
\end{equation}
Here $C_0$ is the same constant as in \eqref{eqn:ls}, and $\Upsilon$ is a non-decreasing bounded function on $(1,\infty)$ that can be rewritten as the probability for a Brownian motion with drift $\alpha -1$ to stay above a random barrier \cite[Theorem 1.2]{BBCM20+}. Similar tight estimates were recently obtained for the upper large deviations of branching random walks \cite{Bura,GH18}.

The aim of this article is to obtain precise lower deviations estimates for the maximum of the branching Brownian motion, i.e.\@ the asymptotic behaviour of the probability $\P(M_t\leq \sqrt{2}\alpha t)$ for all $\alpha<1$. The same question recently popped up in the context of branching random walks \cite{GH18,CH20}. Derrida and Shi \cite{DS16} obtained the following estimates on the exponential decay
\[
  \P(M_t\leq \sqrt{2}\alpha t)=
  \begin{cases}
    e^{-2(\sqrt{2}-1)(1-\alpha )t+o(t)},&\textrm{ for } 1-\sqrt{2}<\alpha<1,\\
    e^{-(1+\alpha^2) t+o(t)},&\textrm{ for }\alpha\leq 1-\sqrt{2},
  \end{cases}
  \quad \text{ as } t \to \infty.
\]
In particular, some transition occurs at $\alpha_c = -\gamma := 1 -\sqrt{2} \approx -0.414$, where the large deviations rate function exhibits a second order phase transition (cf.\@ Figure~\ref{fig:rateFunction}). Derrida and Shi \cite{DS} also conjectured the existence of a positive constant $C^{(1)}$ such that for all $\alpha < 1$,
\begin{equation}
  \label{eqn:dsconjecture}
  \P(M_t\leq \sqrt{2}\alpha t) \sim
  \begin{cases}
    C^{(1)}(\frac{\alpha -\alpha_c}{\sqrt{2}})^{\frac{3\gamma}{2}} t^{\frac{3\gamma}{2}} e^{-2\gamma (1-\alpha)t}, &\textrm{ if }\alpha>\alpha_c,\\
    &\\[-7pt]
    \frac{\Phi(\alpha)}{\sqrt{4\pi}}t^{-\frac{1}{2}}e^{-(1+\alpha^2)t}, &\textrm{ if }\alpha<\alpha_c,
\end{cases}
\end{equation}
where
\begin{equation}
  \label{eqn:defPhi}
  \Phi(\alpha) =: -\frac{1}{\alpha}+\sqrt{2}\int_0^\infty \dd s\int_{\R}\dd y e^{(1-\alpha^2)s+\sqrt{2}\alpha y} u(y,s)^2 \in (0,\infty).
\end{equation}
We shall prove in Section~\ref{max-} that $\Phi(\alpha)$ is finite.

\begin{figure}
\centering
\begin{tikzpicture}[yscale = 0.34]
  \draw[->] (-4,0) -- (4,0) node[above] {$\alpha$};
  \draw[->] (0,-0.6) -- (0,12) node[right] {$\Psi(\alpha)$};
  \draw[blue, thick, domain=1:3.5] plot ({\x}, {(\x*\x-1)});
  \draw (1,-0.2) node[below] {$1$} -- (1,0.2);
  \draw[blue, thick, domain=-0.414:1] plot ({\x}, {2*0.414*(1-\x)});
  \draw (-.414,-0.2) node[below] {$-\gamma$} -- (-.414,0.2);
  \draw[blue, thick, domain=-3.15:-.414] plot ({\x}, {\x*\x+1});
\end{tikzpicture}
\caption{Rate function  for the maximal displacement of the branching Brownian motion, defined as $\Psi(\alpha) = \lim_{t \to \infty} \frac{1}{t} \log \P(M_t \approx \sqrt{2}\alpha t)$ (c.f.\@ \cite{DS}). Note the second order phase transition occurring at position $-\gamma = 1 - \sqrt{2}$.}
\label{fig:rateFunction}
\end{figure}
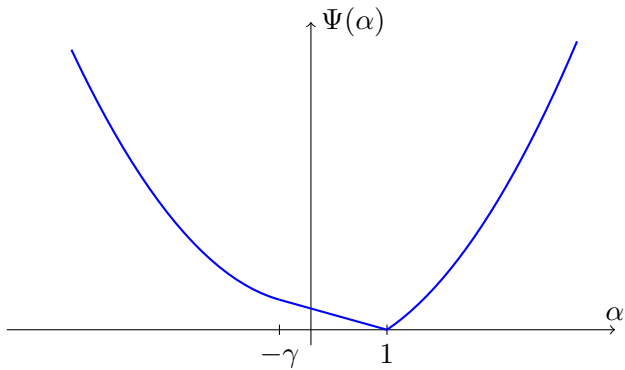

In this work, we prove the conjecture \eqref{eqn:dsconjecture} of Derrida and Shi. Additionally, we obtain the precise asymptotic decay of $\P(M_t \leq \sqrt{2}\alpha t)$ in the critical case $\alpha=1-\sqrt{2}$ as well. We also describe the law of the branching Brownian motion conditioned on the large deviation event $\{M_t\leq \sqrt{2}\alpha t\}$ for all $\alpha<1$, exhibiting the typical behaviour of a branching Brownian motion realizing this large deviation.

This behaviour is governed by the value of $\tau$ the first branching time, defined as the time at which the initial ancestor of the process dies, i.e.
\[\tau:=\inf\{t\geq0: \#N(t)\geq2\}.\]
On the event $\{M_t \leq \sqrt{2}\alpha t\}$, $\tau$ will typically be of order $t$ and, up to some normalization, will converge in distribution as $t \to \infty$. Moreover, the position  $X_\emptyset(\tau)$ at which the initial particle gives birth to children will be, given $\tau$, tight around its median, which is located in a neighbourhood of $-c \tau$ for some $c > 0$.

We also describe, under the probability $\P(\cdot\vert M_t\leq\sqrt{2}\alpha t)$, the asymptotic behaviour of the point measure
\[
\mathcal{E}_t(\alpha):=\sum_{u\in N(t)}\delta_{X_u(t)-\sqrt{2}\alpha t},\quad t\geq0,
\]
which is the extremal process at time $t$ of the conditioned branching Brownian motion, i.e.\@ the position of particles that are within distance $O(1)$ from the maximal position.

The extremal process of the branching Brownian motion (without conditioning) has been previously studied by Arguin, Bovier and Kistler \cite{ABK13} and Aïdékon, Berestycki, Brunet and Shi \cite{ABBS13}. Writing $\mathcal{E}_t:=\sum_{u\in N(t)}\delta_{X_u(t)-m_t}$ the extremal process of the branching Brownian motion, they showed that as $t\to\infty$,
\[
  (\mathcal{E}_t, M_t-m_t)\Longrightarrow(\mathcal{E}, \max_{x\in\mathcal{E}}x),
\]
where $\mathcal{E}$ is a randomly shifted decorated Poisson point process with exponential intensity, and $\Longrightarrow$ denotes convergence in distribution. More precisely, this point process can be constructed as
\[ \mathcal{E} :=\sum_{x\in \mathcal{P}}\sum_{y\in \mathcal{D}_x}\delta_{x+y},\]
where $C_0, D_\infty$ are the quantities defined in \eqref{eqn:ls}, and conditioned on $D_\infty$, $\mathcal{P}$ is a Poisson point process with intensity $C_0\sqrt{2} D_\infty e^{-\sqrt{2}x}\dd x$ and conditioned on $\mathcal{P}$, $(\mathcal{D}_x, x\in\mathcal{P})$ are i.i.d.\@ decorated point processes in $(-\infty,0]$ with an atom at $0$, which we refer to as the decoration of the branching Brownian motion. In particular, observe that
\[
  \max \mathcal{E} = \max \mathcal{P} \egaldistr \frac{1}{\sqrt{2}} \left(G- \log (C_0 \sqrt{2}D_\infty)\right),
\]
where $\max \mathcal{A}$ is the largest position occupied by an atom of the point process $\mathcal{A}$, and $G$ is a random variable independent of $D_\infty$ with standard Gumbel distribution. Therefore the distribution function of $\max \mathcal{E}$ is exactly $w(\cdot)$, in view of \eqref{eqn:ls}.

\begin{remark}
In this article, we choose to focus on branching Brownian motions with binary branching, to keep the proofs as simple as possible. Up to minor changes, one can assume the number of children made by an individual at death to be i.i.d.\@ integer-valued random variables with same law as $L$. As long as $L \geq 1$ a.s.\@ (i.e.\@ the process never gets extinct) and $\E(L (\log L)^2) < \infty$ (an integrability condition guaranteeing the non-degeneracy of the limit $D_\infty$), we expect similar results to hold.
\end{remark}

Before stating our main result, we quickly recall the heuristics given in \cite{DS16} to explain the asymptotic decay of $\P(M_t \leq \sqrt{2}\alpha t)$ in \eqref{eqn:dsconjecture} in the next section.

\subsection{Heuristics behind the conjecture \texorpdfstring{\eqref{eqn:dsconjecture}}{}}

Recall that $\tau$ is the first branching time of the process and $X_\emptyset(\tau)$ the position of the particle at that first branching time. As particles behave independently after they branched, the probability of observing an unusually low maximum decays sharply after each branching event. Therefore, to maximize the possibility that $M_t \leq \sqrt{2}\alpha t$, a good strategy is to make the first branching time as large as possible. Recalling that $\P(\tau > s) = e^{-s}$ and that we expect exponential decay in $t$, it is reasonable to conjecture that $\tau \approx \lambda_\alpha t$ conditioned on $\{M_t \leq \sqrt{2}\alpha t\}$, for some $\lambda_\alpha \in [0,1]$. Additionally, after that branching time, particle should behave as regular branching Brownian motions with length $t-\tau$, therefore the maximal position at time $t$ should be around level $X_\emptyset(\tau) + \sqrt{2}(t-\tau)$, which has to be lower than $\sqrt{2}\alpha t$. We also have the condition $X_\emptyset(\tau) \leq \sqrt{2}\alpha t + \sqrt{2}(\tau-t)$.

Then, with $B$ a standard Brownian motion, observe that
\begin{align*}
  \P(\tau \approx \lambda t, X_\emptyset(\tau) \leq \sqrt{2}\alpha t+\sqrt{2}(\tau -t)) &\approx e^{-\lambda t} \P(B_{\lambda t} \leq  \sqrt{2}\alpha t + \sqrt{2}(\lambda - 1)t)\\
  &\approx \exp\left( - t \left( \lambda + \frac{(\alpha + (\lambda-1))^2}{\lambda} \right) \right).
\end{align*}
Thus, to maximize this probability, one has to choose the parameter $\lambda_\alpha \in [0,1]$ that minimizes the quantity
\[
  \lambda + \frac{(\alpha - (1-\lambda))^2}{\lambda}.
\]
Note that if $\alpha > -\gamma = 1 - \sqrt{2}$, this minimum is attained for $\lambda_\alpha = \frac{(1-\alpha)}{\sqrt{2}} \in [0,1]$, whereas if $\alpha \leq \gamma$, this minimum is attained at $\lambda_\alpha = 1$.

As a result, we expect three different behaviours for the branching Brownian motion conditioned on having a maximum smaller than $\sqrt{2}\alpha t$, depending on whether $\alpha$ is larger than, smaller than, or equal to $-\gamma$. In the first case, the branching time should happen at some intermediate time in the process, and the branching Brownian motion after this first branching time should behave as a regular process, conditioned on an event of positive probability.  If $\alpha < -\gamma$, then one expects the process not to branch until the very end of the process, which allows an explicit description of the extremal process in that case. In the intermediate case $\alpha = - \gamma$, the branching time should be such that $t - \tau$ is large, but negligible with respect to $t$. In this setting, the behaviour of the process after that time should not be too different from the case $\alpha > -\gamma$.

This heuristics describing the lower large deviations for the maximal displacement closely match the one used for lower deviations of similar branching processes. For example, in \cite{vanish}, the law of a Galton-Watson process conditioned on the limiting martingale being small is described as a Galton-Watson process with minimal branching until given generation, then behave as typical process after that generation. Similarly, in \cite{CH20}, a branching random walk with an unusually small minimum is described as a process in which particles make as few children as possible during the first few  branches in as few children as possible, which all drift to a low position, from which they start independent unconditioned branching random walks.

\subsection{Main results}

We now state our main results, which completely validate the above heuristics of \cite{DS16}. With careful analysis of the first branching time and position, we are able to give an equivalent for the lower large deviations of the maximal displacement. We are also able to describe exactly the joint convergence in law of $(\tau, X_\emptyset(\tau), M_t,\mathcal{E}_t(\alpha))$ conditioned on $\{M_t \leq \sqrt{2}\alpha t\}$.

We begin with the case $\alpha > -\gamma$.
\begin{theorem}\label{highmax}
Assume that $\alpha\in(-\gamma,1)$. Then, as $t\to\infty$, we have
\begin{equation}\label{lowerprob+}
\P(M_t\leq\sqrt{2}\alpha t)\sim C^{(1)}(v_\alpha t)^{\frac{3\gamma}{2}}e^{-2\gamma(1-\alpha )t},
\end{equation}
where $C^{(1)}:=\frac{1}{2}\int_\R e^{-\sqrt{2}\gamma z}w(z)^2\dd z\in(0,\infty)$ and $v_\alpha:=\frac{\gamma+\alpha}{\sqrt{2}}\in(0,1)$. Furthermore, conditioned on $\{M_t\leq\sqrt{2}\alpha t\}$,
\begin{equation}\label{jointcvg+}
\left(\frac{\tau-\frac{(1-\alpha)}{\sqrt{2}} t}{\sqrt{t\tfrac{(1-\alpha)}{4\sqrt{2}}}}, X_{\emptyset}(\tau)-(\sqrt{2}\alpha t-m_{t-\tau}), M_t-\sqrt{2}\alpha t\right)\Longrightarrow(\xi, -\chi, -E),
\end{equation}
where $\xi$ and $(\chi, E)$ are independent, with $\xi$ a standard Gaussian random variable and $E$ an exponential variable with parameter $\sqrt{2}\gamma$. The joint distribution of $(\chi, E)$ is given by
\[
  \P(\chi\leq x, E\geq y)=\frac{1}{2C^{(1)}}e^{-\sqrt{2}\gamma y} \int_{-\infty}^{x-y}e^{-\sqrt{2}\gamma z}w(z)^2\dd z, \quad x \in \R,\  y \in \R_+.
\]
Moreover, we have, jointly with the convergence in \eqref{jointcvg+},
\begin{equation}\label{processcvg+}
\mathcal{E}_t(\alpha)\Longrightarrow \mathcal{E}^-:= \sum_{x\in \mathcal{E}_1\cup \mathcal{E}_2}\delta_{x-\chi},
\end{equation}
where given $\chi$, $\mathcal{E}_1$ and $\mathcal{E}_2$ are i.i.d.\@ point processes distributed as $\mathcal{E}$ conditioned on $\{\max\mathcal{E}\leq\chi\}$.
\end{theorem}

\begin{remark}
The finiteness of the constant $C^{(1)}$ defined above can be checked using that $w(z)\sim C e^{\sqrt{2}\gamma z}$ as $z\to-\infty$ (cf. e.g. \cite{ABK11}).
\end{remark}

We now consider the case $\alpha< - \gamma$. In this setting, the total number of particles in the process at time $t$ remains tight, allowing the following description of the process, conditioned on the large deviations event.

\begin{theorem}\label{lowmax}
Assume that $\alpha\in(-\infty,-\gamma)$. Then, as $t\to\infty$, we have
\begin{equation}\label{lowerprob-}
\P(M_t\leq \sqrt{2}\alpha t)\sim  \frac{\Phi(\alpha)}{\sqrt{4\pi}} t^{-\frac{1}{2}}e^{-(1+\alpha^2)t}.
\end{equation}
Moreover, conditioned on $\{M_t\leq\sqrt{2}\alpha t\}$,
\begin{equation}\label{jointcvg-}
\left(t-t\wedge \tau, \sqrt{2}\alpha t-X_\emptyset(t \wedge \tau), M_t - \sqrt{2}\alpha t\right)\Longrightarrow(\xi_\alpha,-\chi_\alpha, -E_\alpha),
\end{equation}
where $\xi_\alpha$ is distributed as
\[
  \frac{1}{-\alpha\Phi(\alpha)}\delta_0(\dd s)+\frac{1}{\Phi(\alpha)}\int_{\R}e^{\sqrt{2}\alpha z+(1-\alpha^2)s}u(z,s)^2\dd z \dd s,
\]
$E_\alpha$ is distributed as an exponential random variable with parameter $-\sqrt{2}\alpha$, and the joint distribution of $(\xi_\alpha,\chi_\alpha, E_\alpha)$ is given by
\begin{multline*}
  \P(\xi_\alpha\leq x_1, \chi_\alpha\leq x_2, E_\alpha\geq x_3)\\
  = \frac{1}{\Phi(\alpha)}\Big(\ind{x_3<x_2}\int_{x_3}^{x_2}\sqrt{2}e^{\sqrt{2}\alpha z}\dd z \qquad\qquad\qquad\qquad\qquad\qquad\\
  +\sqrt{2}\int_0^{x_1} \dd s\int_{-\infty}^{x_2-x_3} e^{\sqrt{2}\alpha (x_3+z)+(1-\alpha^2)s}u(z,s)^2\dd z\Big),
\end{multline*}
for any $x_1, x_3\in\R_+$ and $x_2\in\R$. Further, we have jointly
\begin{equation}\label{processcvg-}
  \mathcal{E}_t(\alpha)\Longrightarrow \mathcal{E}_\infty(\alpha) :=\delta_{-\chi_\alpha} \ind{\xi_\alpha = 0}+\ind{\xi_\alpha > 0}\sum_{x\in \mathcal{B}_1\cup\mathcal{B}_2} \delta_{x-\chi_\alpha},
\end{equation}
where given $(\xi_\alpha, \chi_\alpha)$, $\mathcal{B}_1$ and $\mathcal{B}_2$ are i.i.d.\@ copies of $\sum_{u\in N(\xi_\alpha)}\delta_{X_u(\xi_\alpha)}$ conditioned on $\{M_{\xi_\alpha}\leq \chi_\alpha\}$.
\end{theorem}

\begin{remark}
Recall that $\Phi(\alpha)$ is defined in \eqref{eqn:defPhi}. Observe that the law of $t - t \wedge \tau$ has a Dirac mass at $0$, corresponding to the probability that no branching occurs in the time interval $[0,t]$.
\end{remark}

Theorems~\ref{highmax} and~\ref{lowmax} verify the conjecture of Derrida and Shi, and as expected in the heuristics, the behaviour of the conditioned process is very different depending on the sign of $\alpha + \gamma$. We end up with a description in the boundary case $\alpha = -\gamma$, which except for the asymmetric fluctuations on the time $\tau$ is similar to the case $\alpha > -\gamma$.
\begin{theorem}\label{criticalmax}
If $\alpha=-\gamma$, as $t\to\infty$, we have
\begin{equation}
  \label{lowerprobc}
  \P(M_t\leq \sqrt{2}\alpha t)\sim C^{(2)} t^{3\gamma/4}e^{-(1+\gamma^2)t},
\end{equation}
where
\[
  C^{(2)}:=\frac{1}{\sqrt{2\pi}}\int_{\R_+}u^{3\gamma/2}e^{-2u^2} \dd u\int_\R e^{-\sqrt{2}\gamma z}w(z)^2\dd z
  =\frac{C^{(1)} \Gamma(\frac{3\sqrt{2}-1}{4})}{\sqrt{2\pi}2^{\frac{3\sqrt{2}-1}{4}}}.
\]
Moreover, conditioned on $\{M_t\leq\sqrt{2}\alpha t\}$,
\begin{equation}
  \label{jointcvgc}
  \left(\frac{t-\tau}{\sqrt{t}}, X_\emptyset(\tau)-(\sqrt{2}\alpha t-m_{t-\tau}), M_t-\sqrt{2}\alpha t\right)\Longrightarrow(\xi_{\alpha}, -\chi, -E),
\end{equation}
where $\xi_{\alpha}$ and $(\chi, E)$ are independent, $(\chi, E)$ have same law as in Theorem~\ref{highmax} and $\xi_{\alpha}$ is a positive random variable with density $2^{-3(\sqrt{2}+ 1)/4}\Gamma((3\sqrt{2}- 1)/4)u^{3\gamma/2}e^{-2u^2} \dd u$. Further, we have jointly
\begin{equation}
  \label{processcvgc}
  \mathcal{E}_t(\alpha)\Longrightarrow \mathcal{E}^-,
\end{equation}
where $\mathcal{E}^-$ is the same as in Theorem~\ref{highmax}.
\end{theorem}

We draw in Figure~\ref{fig:recapitule} scheme of the expected behaviour of the branching Brownian motion conditioned to stay below $\sqrt{2}\alpha t$ if $\alpha > -\gamma$ (Theorem~\ref{highmax}), $\alpha = -\gamma$ (Theorem~\ref{criticalmax}), or $\alpha < - \gamma$ (Theorem~\ref{lowmax}).

\begin{remark}
\label{phasetransition}
In fact, we could look closer around the phase transition $\alpha_c=-\gamma$ and obtain the following results by a straightforward adaptation of the reasoning used in Section~\ref{maxc}. We leave the proof to interested readers.
Let $a : \R_+ \to \R$ with $a_t = o(t)$.
\begin{enumerate}
\item If $a_t=o(\sqrt{t})$, then
\begin{equation}\label{lowerprobc+}
  \P(M_t\leq -\sqrt{2}\gamma t+ a_t)\sim C^{(2)} t^{3\gamma/4}e^{-2\sqrt{2}\gamma t+\sqrt{2}\gamma a_t}.
\end{equation}
\item If $a_t=a\sqrt{t}$ with $a\in\mathbb{R}$, there exists a positive function $a \mapsto C(a)$ such that
\begin{equation}\label{lowerprobcc}
  \P(M_t\leq -\sqrt{2}\gamma t+ a_t)\sim C(a) t^{3\gamma/4}e^{-2\sqrt{2}\gamma t+\sqrt{2}\gamma a_t}.
\end{equation}
\item If $\lim_{t\to\infty}\frac{a_t}{\sqrt{t}}=\infty$ and $a_t=o(t)$, then there exists $C^{(3)}, C^{(4)} > 0$ such that
\begin{align}
  \P(M_t\leq -\sqrt{2}\gamma t+ a_t)\sim& C^{(3)} a_t^{3\gamma/2}e^{-2\sqrt{2}\gamma t+\sqrt{2}\gamma a_t}, \label{lowerprobc++}\\
  \P(M_t\leq -\sqrt{2}\gamma t- a_t)\sim&  C^{(4)} (t/a_t)^{3\gamma/2+1}t^{-1/2}e^{-2\sqrt{2} t-\sqrt{2}\gamma a_t-\frac{a_t^2}{4t}}. \label{lowerprobc--}
\end{align}
\end{enumerate}
\end{remark}

\begin{figure}
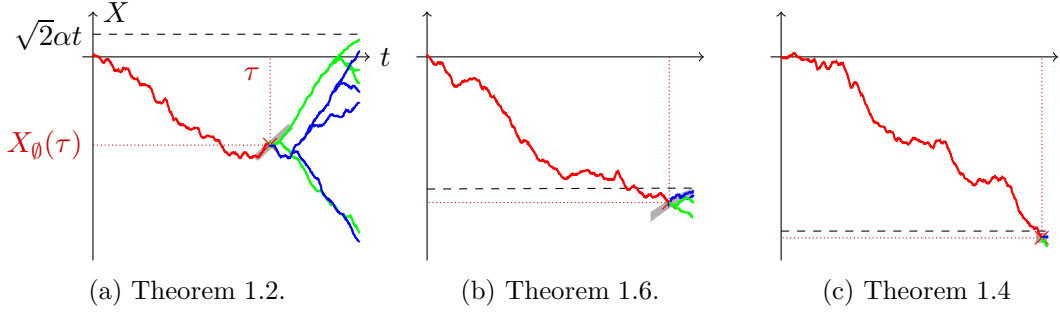

\centering
\begin{subfigure}[t]{0.34\textwidth}
\centering

\caption{Theorem~\ref{highmax}.}
\end{subfigure}
\begin{subfigure}[t]{0.31\textwidth}
\centering
%
\caption{Theorem~\ref{criticalmax}.}
\end{subfigure}
\begin{subfigure}[t]{0.31\textwidth}
\centering
%
\caption{Theorem~\ref{lowmax}}
\end{subfigure}
\caption{Scheme of the first branching time in different conditioning scenarios. The initial particle is figured in red, its two offspring giving birth to the green and blue subtrees respectively. The typical branching zone is figured as a grey area. Note its width is of order $t^{1/2}$ and its height of order $1$ in cases (a) and (b).}
\label{fig:recapitule}
\end{figure}

Finally, we consider the lower moderate deviations for the maximum, i.e.\@ the asymptotic behaviour of the probability of the event $\{M_t\leq m_t-a_t\}$ where $\lim_{t\to\infty}a_t=\infty$ and $a_t=o(t)$. As expected from the heuristic, in that case the first branching time happens at a time of order $a_t$, and the process after that first branching time is a branching Brownian motion conditioned on an event of positive probability. More precisely, the following result holds.
\begin{theorem}
\label{moderatemax}
If $a_t=o(t)$ and $\lim_{t\to\infty}a_t=\infty$, then as $t\to\infty$,
\begin{equation}
  \label{lowerprobm}
  \P(M_t\leq m_t-a_t)\sim C^{(1)} e^{-\sqrt{2}\gamma a_t}.
\end{equation}
Moreover, conditioned on $\{M_t\leq m_t-a_t\}$,
\begin{equation}
  \label{jointcvgm}
  \left(\frac{\tau-\frac{1}{2}a_t}{\sqrt{a_t/8}}, X_\emptyset(\tau)-(\sqrt{2}\tau-a_t), M_t-(m_t-a_t)\right)\Longrightarrow(\xi,-\chi,-E),
\end{equation}
and jointly,
\begin{equation}
  \label{processcvgm}
  \sum_{u\in N(t)}\delta_{X_u(t)-(m_t-a_t)}\Longrightarrow \mathcal{E}^-,
\end{equation}
where $(\xi,\chi,E,\mathcal{E}^-)$ is the same as in Theorem~\ref{highmax}.
\end{theorem}

Note that in this theorem, \eqref{lowerprobm} is already known in the literature (cf. \cite{ABK13}), our contribution consists in the joint convergence in distribution described in (\ref{jointcvgm}--\ref{processcvgm}).

\medskip

The main idea behind the proof of all these results is the decomposition of the branching Brownian motion at its first branching point. More precisely, the cumulative distribution function of the maximal displacement, defined as $u(z,s)=\P(M_s\leq z)$, for all $s \geq 0$ and $z \in \R$, satisfies
\begin{equation}
  \label{key}
  u(z,t)=e^{-t}\P(B_t\le z)+\int_0^t e^{-s}\dd s \int_{\R}\P(B_s\in \dd y) u(z-y,t-s)^2,
\end{equation}
where $(B_t)_{t\geq0}$ is a standard Brownian motion. This formula allows us to bootstrap close to optimal bounds on $u(\sqrt{2}t - a_t,t)$ from a priori bounds, using Laplace's method (cf.\@ e.g.\@~\cite[Chapter~4]{DemboZeitouni}). This allows us to obtain equivalents for different regimes as $t,a_t \to \infty$.

Observe that \eqref{key} is a simple consequence of the Markov property applied at the first branching time of the branching Brownian motion. Indeed, at time $t$, the original ancestor did not split with probability $e^{-t}$, in which case its position is distributed as a Gaussian random variable with variance $t$. Otherwise, the ancestor died at time $s$ with probability $e^{-s} \dd s$, in which case the maximum of the branching Brownian motion at time $t$ has the same law as the maximum of two independent branching Brownian motions at time $t-s$, shifted by the position of the ancestor at time $s$, which is distributed as $B_s$.

We use \eqref{key} to show that with high probability, conditioned on $\{M_t \leq \sqrt{2}t- a_t\}$, the first branching time has to happen at some specific time and position with high probability. Depending on the growth rate of $a_t$, this first branching time can be $o(t)$ (Theorem~\ref{moderatemax}), proportional to $\lambda t$ for some $\lambda \in [0,1]$ (Theorems~\ref{highmax} and~\ref{criticalmax}), or of order $t - o(t)$ (Theorem~\ref{lowmax}).

The rest of the paper is organized as follows. In Section~\ref{facts}, we state some well-known results on branching Brownian motion and show some rough bounds of $u(z,t)$. In Section~\ref{max+}, we treat the case where $\alpha\in(-\gamma,1)$ and prove Theorem~\ref{highmax} and Theorem~\ref{moderatemax} in that context. Section~\ref{max-} is devoted to proving Theorem~\ref{lowmax}. In Section~\ref{maxc}, we consider the critical case and prove Theorem~\ref{criticalmax}. The proofs of some technical lemmas are postponed to the Appendix~\ref{lems}.

In this paper, we write $f(t)\sim g(t)$ as $t\to\infty$ to denote $\lim_{t\to\infty}\frac{f(t)}{g(t)}=1$. As usual, $f(t)=o_t(g(t))$ means $\lim_{t\to\infty}\frac{f(t)}{g(t)}=0$. The quantities $(C_i)_{i\in\N}$ and $(c_i)_{i\in\N}$ represent positive constants, and $c,C$ are non-specified positive constants, that might change from line to line, taken respectively small enough and large enough.

\section{Preliminary results and well-known facts}
\label{facts}

In this section, we recall previously known results on the maximum and extremal process of branching Brownian motions. We source most of the results stated here from the book of Bovier \cite{Bo17} for convenience, and refer the reader to it for the origins of these results. Using these results, we obtain first order estimates on $u(z,t) = \P(M_t \leq z)$.

We begin by recalling that by \cite[Proposition 2.22]{Bo17}: for any family $(\mathcal{D}_t)_{t \in [0,\infty]}$ of point processes, the joint convergence in law $(\mathcal{D}_t,\max \mathcal{D}_t)$ to $(\mathcal{D}_\infty,\max \mathcal{D}_\infty)$ is equivalent to
\begin{equation}
  \label{eqn:Prop222Bov}
  \forall \phi \in \mathcal{C}_c^+(\R), \ \forall z \in \R, \  \lim_{t \to \infty} \E\left[e^{-\int \phi \dd \mathcal{D}_t}; \max \mathcal{D}_t\leq z\right]= \E\left[e^{-\int \phi \dd \mathcal{D}_\infty} ; \max \mathcal{D}_\infty \leq z\right],
\end{equation}
writing $\E\left[X;A\right]$ for $\E\left[X\mathbbm{1}_A\right]$, with $X$ a random variable and $A$ a measurable event.

We denote by $(X_u(t), u \in N(t))_{t \geq 0}$ a standard binary branching Brownian motion. For all $\phi \in \mathcal{C}_c^+(\R)$, we denote by
\begin{equation}
  \label{eqn:defUphi}
  u_\phi : (z,t) \in \R_+ \times \R \mapsto \E\left[e^{-\sum_{u\in N(t)}\phi(X_u(t)-z)}; M_t\leq z\right]=\E\left[\prod_{u\in N(t)}f_\phi(z-X_u(t))\right],
\end{equation}
where we have set $f_\phi:y\mapsto e^{-\phi(-y)}\ind{y\geq0}$. Recall that $u_\phi$ is the unique solution of the F-KPP partial differential equation \eqref{eqn:deffkpp} with initial condition $f_\phi$, i.e.
\begin{equation}
  \label{eqn:fkpp}
  \begin{cases}
  &\partial_t u = \frac{1}{2} \Delta u - u(1-u),\\
  &u_\phi(z,0) = f_\phi(z), \text{ for all } z \in \R.
  \end{cases}
\end{equation}
Remark that the cumulative distribution function of $M_t$ is given by $u(z,t) = u_0(z,t)$.

Then by \eqref{eqn:Prop222Bov}, the joint convergence in law of the centred extremal process and maximal displacement of the branching Brownian motion can be rewritten as the following pointwise convergence
\begin{equation}
  \label{eqn:cvPointwise}
  \forall \phi \in \mathcal{C}_c^+(\R), \forall z \in \R, \lim_{t \to \infty} u_\phi(m_t + z,t) = w_\phi(z),
\end{equation}
where $w_\phi(z):=\E\left[e^{-\int\phi(\cdot-z)\dd \mathcal{E}}; \max \mathcal{E} \leq z\right]$.

Moreover, convergence \eqref{eqn:cvPointwise} in fact holds uniformly on compact sets, by \cite[Lemma~5.5 and Theorem 5.9]{Bo17}. Let $K > 0$, using that $\min_{z\in[-K,K]}w_\phi(z)>0$, this uniform convergence result implies that
\begin{equation}
  \label{uniformcvg}
  \lim_{t \to \infty} \sup_{|z| \leq K} \frac{|u_\phi(m_t+z,t) - w_\phi(z)|}{w_\phi(z)} = 0.
\end{equation}
Applying the above result to the function $\phi \equiv 0$ gives that uniformly on $z \in [-K,K]$, $u(m_t+z,t) = w(z)(1+o(1))$ as $t \to \infty$.

Let $(B_t, t \geq 0)$ be a standard Brownian motion. We recall the following classical asymptotic on the tail of the standard Gaussian variable (cf. e.g. \cite[Lemma 1.1]{Bo17}). For any $z > 0$,
\begin{equation}\label{base}
\frac{1}{z\sqrt{2\pi}}e^{-z^2/2}(1-2z^{-2})\leq\P(B_1>z)\leq \frac{1}{z\sqrt{2\pi}}e^{-z^2/2}.
\end{equation}
It follows immediately that for any $z>0$ and $t>0$,
\begin{equation}\label{basic}
\int_z^\infty \frac{e^{-\frac{y^2}{2t}}}{\sqrt{2\pi t}}\dd z=\int_{-\infty}^{-z}\frac{e^{-\frac{y^2}{2t}}}{\sqrt{2\pi t}}\dd z=\P(B_t>z)\leq \frac{\sqrt{t}}{z\sqrt{2\pi}}e^{-\frac{z^2}{2t}}.
\end{equation}

Observe that \eqref{uniformcvg} gives tight bounds on $u(z,t)$ for $z$ in a neighbourhood of $m_t$. We use the above equation \eqref{basic} to give cruder bounds on $u(z,t)$ outside of this neighbourhood. At time $t$, the system contains $\#N(t)\geq 1$ individuals, the position of which is distributed with same law as $B_t$. Therefore, for any $t\geq 0$ and $z\in\R$,
\[
u(z,t)=\P(M_t\leq z)\leq \P(B_t\leq z),
\]
which using \eqref{basic}, yields for $z <0$:
\begin{equation}\label{upbdB}
  u(z,t)\le \frac{\sqrt{t}}{-z\sqrt{2\pi}}e^{-\frac{z^2}{2t}}.
\end{equation}
This straightforward upper bound, combined with the lower deviation results of Derrida and Shi \cite{DS16} gives us the following lemma.

\begin{lemma}\label{DSbdu}
For any $\beta \geq 1$ and $\epsilon > 0$, there exists $t_{\epsilon,\beta} > 1$ such that for any $t\geq t_{\epsilon,\beta}$,
\begin{equation}\label{DSupbdu}
  u(\sqrt{2}at,t)\leq
  \begin{cases}
    1, &\textrm{ if } a\geq 1;\\
    e^{-2\gamma(1-a)t+\epsilon t}, &\textrm{ if } -\gamma \leq a<1;\\
    e^{-(1+a^2)t+\epsilon t}, &\textrm{ if } -\beta\leq a< - \gamma ;\\
    e^{-a^2t}, &\textrm{ if } a<-\beta.
  \end{cases}
\end{equation}
\end{lemma}

\begin{proof}
Let $\beta \geq 1$, we begin by noting that $u(z,t) \leq 1$ for any $z \in \R$ and $t \geq 0$. Additionally, by \eqref{upbdB}, for any $a < -1$, we have
\begin{equation}
  \label{roughupbdu}
  u(\sqrt{2}at,t) \leq \frac{1}{-2a \sqrt{\pi t}} e^{- a^2 t} \leq e^{-a^2 t},
\end{equation}
for all $t \geq 1$. To complete the proof, it is therefore enough to bound $u(\sqrt{2} a t, t)$ for $a \in [-\beta,1)$.

We first reformulate Derrida ans Shi's result \cite[Theorem 1]{DS16} as follows:
\begin{equation}
  \label{DS17M}
  \lim_{t\to\infty}\frac{1}{t} \log u(\sqrt{2} a t, t)=\psi(a):=
  \begin{cases}
    0, &\textrm{ if } a\geq 1;\\
    -2\gamma(1-a), &\textrm{ if }-\gamma \leq a<1;\\
    -(1+a^2), &\textrm{ if } a<-\gamma.
  \end{cases}
\end{equation}
Note that being a cumulative distribution function for any $t \geq 0$, the function $z\mapsto u(z,t)$ is non-decreasing. Thus, both $\frac{\log u(\sqrt{2} a t, t)}{t}$ and $\psi(a)$ are non-decreasing in $a\in\R$, and moreover $\psi$ is continuous. By Dini's theorem, the convergence in \eqref{DS17M} holds uniformly on any compact sets in $\mathbb{R}$, hence in particular on $[-\beta,1]$. As a result, for all $\epsilon>0$ ,there exists $t_{\epsilon,\beta} > 1$ such that for all $t \geq t_{\epsilon,\beta}$, we have
\begin{equation}\label{uniformbd}
\sup_{a\in[-\beta,1]}\Big\vert \frac{1}{t}\log u(\sqrt{2} a t, t) -\psi(a)\Big\vert \leq \epsilon.
\end{equation}
We then deduce \eqref{DSupbdu} from \eqref{uniformbd} and \eqref{roughupbdu}.
\end{proof}

Next, we recall \cite[Theorem 1.7]{CH20}, that gives a tight estimate on the moderate lower deviations of the maximal displacement: for any sequence $(a_t)$ such that $\lim_{t\to\infty}a_t =\infty$ and $a_t = o(t)$,
\[
  \P(M_t\leq m_t-a_t)=e^{-\sqrt{2}(\gamma+o_t(1))a_t}.
\]
To complete this section, we strengthen the above estimate into the following non-asymptotic upper bound for $u(m_t-z,t)$.
\begin{lemma}\label{CHbdu}
For any $\delta\in(0,1)$, there exist $K_\delta\geq 1$ and $T_\delta\geq1$, such that for any $t\geq T_\delta$ and any $z\geq K_\delta$,
\begin{equation}\label{CHupbdu}
u(m_t-z,t)=\P(M_t\leq m_t-z)\leq c_\delta e^{-\sqrt{2}\gamma(1-\delta)z},
\end{equation}
with $c_\delta>1$ a constant depending on $\delta$.
\end{lemma}

The idea of the proof of this result is mainly borrowed from the proof of Theorem~3.2 (Case 2) in \cite{GH18}. We apply the Markov property at some intermediate time, and observe that either there is an anomalously small number of particles alive at that time, or all of the particles alive at that time must satisfy Lemma~\ref{DSbdu}. The detailed proof is postponed to Appendix~\ref{lems1}.

\section{The case \texorpdfstring{$-\gamma<\alpha<1$}{alpha larger than gamma}: proof of Theorems~\ref{highmax} and~\ref{moderatemax}}\label{max+}

In this section, we treat the case when $1-\sqrt{2}<\alpha<1$ and  prove Theorem~\ref{highmax}. The proof of Theorem~\ref{moderatemax}, which can be though of as $\alpha = 1 - o_t(1)$ can be obtained in a very similar fashion. We thus feel free to omit it.

Recall that if $-\gamma <\alpha < 1$, we expect the existence of $\lambda_\alpha \in (0,1)$ such that with high probability the first branching time of the branching Brownian motion conditioned on the event $\{M_t \leq \sqrt{2}\alpha t\}$ is close to $\lambda_\alpha t$. In this situation, we have $\lambda_\alpha = \frac{1-\alpha}{\sqrt{2}}$.

Let $\phi \in \mathcal{C}_c^+(\R)$, applying the Markov property at the first branching time $\tau$, we obtain that $u_\phi$ satisfies \eqref{key} as well:
\begin{align}
  \label{eqn:keyPhi}
  u_\phi(z,t) &= e^{-t} \E\left( e^{-\phi(B_t-z)}; B_t \leq z\right) + \int_0^t e^{-s} \dd s \int_\R \P(B_s \in \dd y) u_\phi(z-y,t-s)^2\\
  &=: U_1^\phi (z,t)+U_2^\phi(z,t).\label{key+}
\end{align}
Note that $U_1^\phi(z,t)$ is the contribution to $u_\phi(z,t) = \E\left(e^{-\sum_{u \in N(t)}\phi(X_u(t)-z)}; M_t \leq z\right)$, which comes from the event $\{\tau > t\}$ on which no branching occurred. As we write $u$ for $u_0$, we denote by $U_1$ and $U_2$ the quantities defined above with $\phi \equiv 0$.

Using standard computations for the Brownian motion, we first give an uniform estimate of $U^\phi_1(\sqrt{2}\alpha t, t)$ for $\alpha \in [0,1)$, as well as an exact asymptotic for $\alpha<0$.
\begin{lemma}\label{lemcvgI1}
Let $\phi \in \mathcal{C}_c^+(\R)$, for $\alpha \geq 0$, we have
\begin{equation}\label{bdI1}
\frac{e^{-\Vert \phi \Vert_\infty}}{2}e^{-t} \leq U^\phi_1(\sqrt{2}\alpha t, t)\leq e^{-t}.
\end{equation}
Moreover, for $\alpha<0$, we have
\begin{equation}\label{cvgI1}
\lim_{t\to\infty}\sqrt{t}e^{(1+\alpha^2)t}U^\phi_1(\sqrt{2}\alpha t, t)=\frac{1}{\sqrt{2\pi}} \int_{-\infty}^0 e^{- \phi(y) - \sqrt{2}\alpha y} \dd y.
\end{equation}
In particular, $U_1(\sqrt{2}\alpha t, t) \sim \frac{e^{-(1+\alpha^2)t}}{\sqrt{4\pi t}|\alpha|}$ as $t \to \infty$.
\end{lemma}

\begin{proof}
We have $U_1^\phi(\sqrt{2}\alpha t, t) = e^{-t} \E(e^{-\phi(B_t - \sqrt{2}\alpha t)}; B_t \leq \sqrt{2}\alpha t)$. For all $\alpha \geq 0$, as $\phi$ is non-negative, we have
$
  1 \geq \E(e^{-\phi(B_t - \sqrt{2}\alpha t)}; B_t \leq \sqrt{2}\alpha t) \geq e^{-\Vert \phi \Vert_\infty} \P(B_t \leq 0),
$
which is enough to prove \eqref{bdI1}.

Additionally, for $\alpha < 0$, by Girsanov transform, we then have
\begin{align*}
  U_1^\phi(\sqrt{2}\alpha t, t)
  &= e^{-(1+\alpha^2)t} \E\left( e^{-\phi(B_t) - \sqrt{2}\alpha B_t} ; B_t \leq 0 \right)\\
  &= \frac{e^{-(1+\alpha^2)t}}{\sqrt{2 \pi t}} \int_{-\infty}^0 e^{-y^2/2t - \phi(y) - \sqrt{2}\alpha y} \dd y.
\end{align*}
By the dominated convergence theorem, it yields
\[\lim_{t \to \infty} \int_{-\infty}^0 e^{-y^2/2t - \phi(y) - \sqrt{2}\alpha y} \dd y = \int_{-\infty}^0 e^{- \phi(y) - \sqrt{2}\alpha y} \dd y,\]
which completes the proof.
\end{proof}

\begin{remark}
For all $\alpha\in(-\gamma,1)$, we have $u(\sqrt{2}\alpha t, t)=e^{-2\gamma(1-\alpha)t(1+o_t(1))}$ by~\cite{DS16}. Moreover, observe that
\[
  2\gamma(1-\alpha)<
  \begin{cases}
    1, &\textrm{ if } 0\leq\alpha<1;\\
    1+\alpha^2 , &\textrm{ if }-\gamma<\alpha<0.
  \end{cases}
\]
As a result, Lemma~\ref{lemcvgI1} shows that $U_1(\sqrt{2}\alpha t, t)=o_t(1)u(\sqrt{2}\alpha t, t)$ for all $\alpha\in(-\gamma,1)$, which by \eqref{key+} implies that $u(\sqrt{2}\alpha t,t) \sim U_2(\sqrt{2}\alpha t,t)$ as $t \to \infty$.
\end{remark}

We thus turn to study $U^\phi_2(\sqrt{2}\alpha t, t)$, which is defined, by a es, as
\[
  U^\phi_2(\sqrt{2}\alpha t, t) = \int_0^t \dd s \int_\R \dd z \frac{e^{-(t-s)-\frac{(m_s+z-\sqrt{2}\alpha t)^2}{2(t-s)}}}{\sqrt{2\pi (t-s)}}u_\phi(m_s+z,s)^2.
\]
Recalling that $v_\alpha = \frac{\gamma+\alpha}{\sqrt{2}} = 1-\frac{1-\alpha}{\sqrt{2}}\in (0,1)$, we begin by observing that most of the mass on this double integral is carried by $\{(s,y): |s-v_\alpha t| \leq A \sqrt{t}, |y-m_s| \leq K\}$, with $A,K$ large enough constants. This is consistent with Theorem~\ref{highmax}, and can be thought of as a proof of the tightness of the family of variables
\[\left\{(t^{-1/2}(\tau - (1 - v_\alpha) t), X_\emptyset(\tau) - (\sqrt{2}\alpha t- m_{t-\tau}),M_t-\sqrt{2}\alpha t, \mathcal{E}_t), t \geq 0\right\}.\]
For any Borel sets $I \subset [0,t]$ and $B \subset \R$, let
\[
  U^\phi_2(\sqrt{2}\alpha t, t,I,B):=\int_I\dd s\int_B  \dd z \frac{e^{-(t-s)-\frac{(m_s+z-\sqrt{2}\alpha t)^2}{2(t-s)}}}{\sqrt{2\pi (t-s)}}u_\phi(m_s+z,s)^2.
\]

\begin{lemma}\label{badtime+badposition}
Let $\alpha \in (-\gamma,1)$, we set
$
  I_{t,A} = \left[ v_\alpha t - A \sqrt{t} , v_\alpha t + A \sqrt{t} \right] \cap [0,t]
$ for all $A, t > 0$. 
For all $\phi \in \mathcal{C}_c^+(\R)$, we have
\begin{equation}\label{badtimecvg}
  \limsup_{K \to \infty} \limsup_{t \to \infty} \frac{e^{2\gamma(1-\alpha)t}}{t^{3\gamma/2}} \left[U^\phi_2(\sqrt{2}\alpha t, t)-U^\phi_2(\sqrt{2}\alpha t, t,I_{t,A},[-K,K])\right] =o_{A}(1).
\end{equation}
\end{lemma}

The proof of Lemma~\ref{badtime+badposition} is postponed to Appendix~\ref{lems2}. A consequence of this result is that the asymptotic behaviour of $U^\phi_2(\sqrt{2}\alpha t, t)$ as $t \to \infty$ is captured by the following lemma, that is used to complete the proof of Theorem~\ref{highmax}.
\begin{lemma}\label{goodtime+position}
Let $\alpha \in (-\gamma,1)$, we set
$
  I_{t,a,b} = \left[ v_\alpha t + a \sqrt{t} , v_\alpha t + b\sqrt{t} \right] \cap [0,t]
$  for all $a < b \in \R$.
Then for all $a<b$ and $c < d$, we have
\begin{equation}\label{goodtimecvg}
  \lim_{t\to\infty}\frac{e^{2\gamma(1-\alpha)t}}{(v_\alpha t)^{3\gamma/2}}U_2(\sqrt{2}\alpha t, t,I_{t,a,b},[c,d])\\
  =\int_{a}^b\frac{e^{-\frac{2\sqrt{2}r^2}{1-\alpha}}}{\sqrt{2\pi\frac{1-\alpha}{\sqrt{2}}}}\dd r\int_{c}^de^{-\sqrt{2}\gamma}w_\phi(z)^2\dd z.
\end{equation}
\end{lemma}

\begin{proof}
Recall that we can write
\[
  U^\phi_2(\sqrt{2}\alpha t, t, I_{t,a,b},[c,d]) =
  \int_{v_\alpha t+ a\sqrt{t}}^{v_\alpha t+b\sqrt{t}}\dd s \frac{e^{-(t-s)}}{\sqrt{2\pi (t-s)}}\int_{c}^d e^{-\frac{(\sqrt{2}\alpha t-m_{s}-z)^2}{2(t-s)}}u_\phi(m_{s}+z,s)^2\dd z.
\]
By the uniform convergence \eqref{uniformcvg}, we observe that uniformly in $s \in I_{t,a,b}$,
\[
  \int_{c}^d e^{-\frac{(\sqrt{2}\alpha t-m_{s}-z)^2}{2(t-s)}}u_\phi(m_{s}+z,s)^2\dd z \sim \int_{c}^d e^{-\frac{(\sqrt{2}\alpha t-m_{s}-z)^2}{2(t-s)}}w(z)^2\dd z,
\]
as $t \to \infty$. Then, with the change of variable $s=ut$, we have
\begin{multline*}
  U_2(\sqrt{2}\alpha t, t,I_{t,a,b},[c,d])\\
  \sim \int_{v_\alpha+\frac{a}{\sqrt{t}}}^{v_\alpha+\frac{b}{\sqrt{t}}} \dd u \frac{te^{-tg_\alpha(u)}}{\sqrt{2\pi t(1-u)}}\int_{c}^d e^{\frac{u-\alpha}{1-u}(\frac{3}{2}\log(ut)-\sqrt{2}z)-\frac{\left(\frac{3}{2\sqrt{2}}\log(ut)-z\right)^2}{2t(1-u)}}w(z)^2\dd z,
\end{multline*}
as $t \to \infty$, by setting
\begin{equation}
  \label{galpha}
  g_\alpha(u):u \in (0,1) \mapsto (1-u)+\frac{(\alpha-u)^2}{1-u}.
\end{equation}
Note that uniformly in $z\in[c,d]$ and in $u\in[v_\alpha+\frac{a}{\sqrt{t}},v_\alpha+\frac{b}{\sqrt{t}}]$, as $t\to\infty$ we have
\begin{align*}
  &\frac{\left(\frac{3}{2\sqrt{t}}\log(ut)-z\right)^2}{2t(1-u)}=o_t(1)\\
  \textrm{ and }\quad &\frac{u-c}{1-u}\left(\frac{3}{2}\log(ut)-\sqrt{2}z\right)=\frac{3\gamma}{2}\log(v_\alpha t)-\sqrt{2}\gamma z+o_t(1).
\end{align*}
It then follows that as $t \to \infty$,
\begin{equation}
  \label{eqn:intermediary}
  U_2(\sqrt{2}\alpha t, t,I_{t,a,b},[c,d])
  \sim \frac{(v_\alpha t)^{3\gamma/2}}{\sqrt{2\pi (1-v_\alpha)}}\int_{v_\alpha+\frac{a}{\sqrt{t}}}^{v_\alpha+\frac{b}{\sqrt{t}}}  e^{-tg_\alpha(u)}\sqrt{t}\dd u\int_{c}^d e^{-\sqrt{2}\gamma z}w(z)^2\dd z.
\end{equation}
We estimate that quantity by doing an asymptotic expansion of $g_\alpha$ around $v_\alpha$.

By change of variable $r=\sqrt{t}(u-v_\alpha)$, we have
\[
  \int_{v_\alpha+\frac{a}{\sqrt{t}}}^{v_\alpha+\frac{b}{\sqrt{t}}}  e^{-tg_\alpha(u)}\sqrt{t}\dd u = \int_{a}^be^{-tg_\alpha(v_\alpha+\frac{r}{\sqrt{t}})}\dd r.
\]
We note that $g_{\alpha}$ is smooth and strictly convex, and attains its minimum of $2 \gamma (1-\alpha)$ at $u = v_\alpha$. By Taylor's expansion at $v_\alpha$, we have as $|h|\downarrow0$,
\begin{equation}\label{Taylorg}
g_\alpha(v_\alpha+h)-g_\alpha(v_\alpha)=g'(v_\alpha)h+\frac{1}{2}g''(v_\alpha) h^2+o(h^2)=\frac{2\sqrt{2}}{1-\alpha}h^2+o(h^2).
\end{equation}
Hence,
\begin{align*}
  \int_{v_\alpha+\frac{a}{\sqrt{t}}}^{v_\alpha+\frac{b}{\sqrt{t}}}  e^{-tg_\alpha(u)}\sqrt{t}\dd u &= e^{-2\gamma (1-\alpha)t}\int_{a}^{b} e^{-\frac{2\sqrt{2}}{1-\alpha}r^2 + o_t(1)} \dd r 
  \sim e^{-2\gamma (1-\alpha)t}\int_{a}^{b} e^{-\frac{2\sqrt{2}}{1-\alpha}r^2 } \dd r,
\end{align*}
as $t \to \infty$ by dominated convergence. In view of \eqref{eqn:intermediary}, this is enough to complete the proof of \eqref{goodtimecvg}.
\end{proof}

We are now ready to prove Theorem~\ref{highmax}.
\begin{proof}[Proof of Theorem~\ref{highmax}]
For all $\phi \in \mathcal{C}_c^+(\R)$, $x_1,x_2 \in \R$ and $x_3 \geq 0$, we set
\begin{multline*}
  F_t(\phi;x_1,x_2,x_3)\\
  := \E\left( e^{-\int \phi \dd \mathcal{E}_t(\alpha)} ; \tfrac{\tau - (1-v_\alpha)t}{\sqrt{t}}\leq  x_1, X_\emptyset(\tau)\geq \sqrt{2}\alpha t-m_{t-\tau}-x_2, M_t\leq \sqrt{2}\alpha t-x_3 \right),
\end{multline*}
and we shall study the asymptotic behaviour of this quantity as $t \to \infty$. Applying the Markov property at time $\tau$, we have
\[
  F_t(\phi;x_1,x_2,x_3) = U_2^{\tau_{x_3}\phi}\left(\sqrt{2}\alpha t, t, \left[v_\alpha t - x_1 \sqrt{t}, t\right],(-\infty,x_2]\right),
\]
with $\tau_{x_3}(\phi) : y \mapsto \phi(y - x_3)$. Therefore, using Lemma~\ref{badtime+badposition} gives
\begin{multline*}
  \lim_{t \to \infty}e^{2\gamma(1-\alpha)t}(v_\alpha t)^{-3\gamma/2} F_t(\phi;x_1,x_2,x_3)\\ = \int_{-x_1}^A\frac{e^{-\frac{2\sqrt{2}r^2}{1-\alpha}}}{\sqrt{2\pi\frac{1-\alpha}{\sqrt{2}}}}\dd r\int_{-K}^{x_2} e^{-\sqrt{2}\gamma z}w_\phi(z-x_3)^2\dd z +o_A(1) + o_K(1),
\end{multline*}
with the $o_A(1)$ term being uniform in $K$, using that by definition, $w_{\tau_x \phi} = w_\phi(\cdot - x)$. Hence, letting $K \to \infty$ then $A \to \infty$, we conclude that
\begin{equation}
  \label{eqn:masterHigh2}
  \lim_{t \to \infty} \frac{e^{2\gamma(1-\alpha)t}}{(v_\alpha t)^{3\gamma/2}} F_t(\phi;x_1,x_2,x_3)
  =\int_{-x_1}^\infty\frac{e^{-\frac{2\sqrt{2}r^2}{1-\alpha}}}{\sqrt{2\pi\frac{1-\alpha}{\sqrt{2}}}}\dd r\int_{-\infty}^{x_2} e^{-\sqrt{2}\gamma z}w_\phi(z-x_3)^2\dd z.
\end{equation}
Using this result, we can now complete the proof of Theorem~\ref{highmax}.

We begin by proving \eqref{lowerprob+}. By \eqref{key}, we have
\[
  \P(M_t \leq \sqrt{2}\alpha t) = U_1(\sqrt{2}\alpha t, t) + U_2(\sqrt{2}\alpha t, t) =  U_2(\sqrt{2}\alpha t, t) + o(t^{3\gamma/2} e^{-(1+\alpha^2)t}),
\]
using Lemma~\ref{lemcvgI1}. Applying then Lemma~\ref{badtime+badposition}, for all $A, K > 0$ we have
\begin{multline*}
  \lim_{t \to \infty} (v_\alpha t)^{-3\gamma/2}e^{(1+\alpha^2)t}\P(M_t \leq \sqrt{2}\alpha t) \\
  = \lim_{t \to \infty} (v_\alpha t)^{-3\gamma/2}e^{(1+\alpha^2)t} F_t(0;A,K,0) + o_A(1) + o_K(1).
\end{multline*}
Hence, letting $K \to \infty$ then $A \to \infty$, by monotone convergence theorem, \eqref{eqn:masterHigh2} yields
\[
  \lim_{t \to \infty} (v_\alpha t)^{-3\gamma/2}e^{(1+\alpha^2)t}\P(M_t \leq \sqrt{2}\alpha t) = \int_\R \frac{e^{-\frac{2\sqrt{2}r^2}{1-\alpha}}}{\sqrt{2\pi\frac{1-\alpha}{\sqrt{2}}}}\dd r\int_\R  e^{-\sqrt{2}\gamma z}w(z)^2\dd z = C^{(1)}.
\]

We now turn to the proof of \eqref{jointcvg+}. By \eqref{eqn:masterHigh2}, for all $x_1,x_2\in\R$ and $x_3\in\R_+$, we have
\begin{multline*}
\P\left(\tau\leq \frac{1-\alpha}{\sqrt{2}}t+ x_1\sqrt{t}, X_\emptyset(\tau)\geq \sqrt{2}\alpha t-m_{t-\tau}-x_2, M_t\leq \sqrt{2}\alpha t-x_3\right)\\
=F_t(0;x_1,x_2,x_3) \sim (v_\alpha t)^{3\gamma/2}e^{-(1+\alpha^2)t} \int_{-x_1}^\infty\frac{e^{-\frac{2\sqrt{2}r^2}{1-\alpha}}}{\sqrt{2\pi\frac{1-\alpha}{\sqrt{2}}}}\dd r\int_{-\infty}^{x_2} e^{-\sqrt{2}\gamma z}w(z-x_3)^2\dd z,
\end{multline*}
which we can rewrite
\[
  \lim_{t \to \infty} \frac{F_t(0;x_1,x_2,x_3)}{\P(M_t \leq \sqrt{2}\alpha t)} = \frac{1}{2C^{(1)}} \int_{-\infty}^{x_1} \dd r \frac{e^{-\frac{2\sqrt{2}r^2}{1-\alpha}}}{\sqrt{2\pi\frac{1-\alpha}{4\sqrt{2}}}} \times e^{-\sqrt{2}\gamma x_3}\int_{-\infty}^{x_2-x_3}e^{-\sqrt{2}\gamma z}w(z)^2\dd z,
\]
which completes the proof of \eqref{jointcvg+}.

We finally turn to the proof of \eqref{processcvg+}, i.e.\@ the joint convergence in distribution of the extremal process seen from $\sqrt{2}\alpha t$, conditioned on $\{M_t \leq \sqrt{2}\alpha t\}$. By a straightforward adaptation of \cite[Proposition~2.2]{Bo17}, to obtain this weak convergence, it is enough to obtain for all $\phi \in \mathcal{C}_c^+(\R)$ and $x_1,x_2\in \R$, $x_3 \geq 0$ the convergence
\begin{multline*}
  \lim_{t \to \infty} \E\left[e^{-\int \phi \dd \mathcal{E}_t(\alpha)}; \tau\leq (1-v_\alpha)t+ x_1\sqrt{t}, X_\emptyset(\tau)\geq \sqrt{2}\alpha t-m_{t-\tau}-x_2\middle| M_t\leq \sqrt{2}\alpha t\right]\\
   = \E\left[ e^{-\int \phi \dd \mathcal{E}^-}; \chi \leq x_2\right] \P\left(\xi \leq x_1\sqrt{\tfrac{4\sqrt{2}}{1-\alpha}}\right).
\end{multline*}
By \eqref{eqn:masterHigh2} and \eqref{processcvg+}, we have immediately that
\[
  \lim_{t \to \infty} \frac{F(\phi,x_1,x_2,0)}{\P(M_t \leq \sqrt{2}\alpha t)} = \frac{1}{2C^{(1)}} \int_{-\infty}^{x_1} \dd r \frac{e^{-\frac{2\sqrt{2}r^2}{1-\alpha}}}{\sqrt{2\pi\frac{1-\alpha}{4\sqrt{2}}}} \times\int_{-\infty}^{x_2}e^{-\sqrt{2}\gamma z}w_\phi(z)^2\dd z.
\]
Observe that according to the definition of $\mathcal{E}^-$, writing $\mathcal{E}$ the limiting extremal process of the unconditioned branching Brownian motion, we have
\begin{align*}
  \E\left[  e^{-\int \phi \dd \mathcal{E}^-}; \chi \leq x_2 \right] &= \int_{-\infty}^{x_2} \E\left[e^{-\sum_{x\in\mathcal{E}}\phi(x-z)}\middle\vert \max \mathcal{E}\leq z\right]^2\P(\chi\in\dd z)\\
  &=  \frac{1}{2C^{(1)}} \int_{-\infty}^{x_2}e^{-\sqrt{2}\gamma z}w_\phi(z)^2\dd z,
\end{align*}
which is therefore enough to end the proof.
\end{proof}

Theorem~\ref{moderatemax} is obtained following a similar line of proof as Theorem~\ref{highmax}. The principal difference is that the Laplace method in the proof of Lemma~\ref{goodtime+position} has to be applied with a maximum obtained on the boundary of the interval of definition. All other estimates follow with straightforward modifications, by replacing $1-\alpha$ by $a_t/\sqrt{2}t$.

\section{The case \texorpdfstring{$\alpha<-\gamma$}{alpha smaller than gamma}: proof of Theorem~\ref{lowmax}}\label{max-}

We now treat the case of $\alpha < - \gamma$. We use in this section that, conditioned on the event $\{M_t \leq \sqrt{2}\alpha t\}$, with high probability no branching occurs before time $t-O(1)$. We use this observation to prove Theorem~\ref{lowmax}, using the same decomposition of $u_\phi(t,x)$ as $U^\phi_1 + U^\phi_2$ as in the previous section. Contrarily to the previous section, $U_1^\phi$ and $U_2^\phi$ are of the same order of magnitude.

Note that the asymptotic behaviour of $U^\phi_1(\sqrt{2}\alpha t,t)$ is given by Lemma~\ref{lemcvgI1}. To study the asymptotic behaviour of $U^\phi_2(\sqrt{2}\alpha t, t)$, we begin by observing that $t-\tau$ is tight conditioned on $\{M_t \leq \sqrt{2}\alpha t\}$.
\begin{lemma}\label{largetime}
Assume that $\alpha<-\gamma$, then for all $\phi \in \mathcal{C}_c^+(\R)$ we have
\begin{equation}\label{upbdlargetime}
  \lim_{A\to\infty}\lim_{t\to\infty}\sqrt{t}e^{(1+\alpha^2)t}U_2^\phi(\sqrt{2}\alpha t, t, [A, t],\R)=0.
\end{equation}
\end{lemma}

The proof of Lemma~\ref{largetime} is postponed to Appendix~\ref{lems3}. The next lemma completes the description of the asymptotic of $U_2^\phi$.

\begin{lemma}\label{smalltime}
If $\alpha<-\gamma$, then for any $x>0$ and $-\infty \leq c < d \leq \infty$ we have
\begin{equation}\label{cvgsmalltime}
\lim_{t\to\infty}\sqrt{t}e^{(1+\alpha^2)t}U^\phi_2(\sqrt{2}\alpha t, t,[0,x],(c,d))=\int_0^x \int_c^d e^{\sqrt{2}\alpha y}u_\phi(y,s)^2 e^{(1-\alpha^2)s} \dd y\dd s.
\end{equation}
Moreover, we have
\begin{equation}\label{Phifini}
\int_0^\infty \int_\R e^{(1-\alpha^2)s+\sqrt{2}\alpha y}u(y,s)^2\dd y \dd s <\infty.
\end{equation}
\end{lemma}

\begin{proof}
Observe that we can rewrite
\begin{align*}
U_2^\phi(\sqrt{2}\alpha t, t, [0, x],[c,d])=&\int_0^x\dd s\int_c^d\dd y \frac{e^{-(t-s)-\frac{(\sqrt{2}\alpha t-y)^2}{2(t-s)}}}{\sqrt{2\pi (t-s)}}u_\phi(y,s)^2\\
=&\frac{e^{-(1+\alpha^2)t}}{\sqrt{2\pi t}}\int_0^x \dd s \int_c^d \dd y \sqrt{\frac{t}{t-s}} e^{(1-\alpha^2)s}e^{\sqrt{2}\alpha y-\frac{(\sqrt{2}\alpha s-y)^2}{2(t-s)}}u_\phi(y,s)^2\\
\sim&\frac{e^{-(1+\alpha^2)t}}{\sqrt{2\pi t}}\int_0^x \dd s \int_c^d \dd y e^{(1-\alpha^2)s} e^{\sqrt{2}\alpha y-\frac{(\sqrt{2}\alpha s-y)^2}{2(t-s)}}u_\phi(y,s)^2,
\end{align*}
as $t \to \infty$. Then, by monotone convergence theorem, as $t \to \infty$ we have
\[
  \int_0^x \int_c^d e^{\sqrt{2}\alpha y-\frac{(\sqrt{2}\alpha s-y)^2}{2(t-s)}}u_\phi(y,s)^2e^{(1-\alpha^2)s} \dd y \dd s\rightarrow \int_0^x \int_c^d e^{\sqrt{2}\alpha y}u_\phi(y,s)^2e^{(1-\alpha^2)s} \dd y\dd s,
\]
which completes the proof of \eqref{cvgsmalltime}.

The rest of the proof is devoted to show that $\int_0^\infty \int_\R e^{(1-\alpha^2)s+\sqrt{2}\alpha y}u(y,s)^2\dd s \dd y < \infty$. As a first step, we bound for any $s \geq 0$ the quantity $I_s := \int_\R e^{\sqrt{2}\alpha y} u(y,s)^2 \dd y$. First, by \eqref{upbdB}, for all $y \leq 0$ we have $0 \leq u(y,s)\leq \frac{s^{1/2}}{|y|} e^{-y^2/2s} \wedge 1$, therefore
\begin{align*}
  I_s &\leq \int_{-\infty}^1  e^{\sqrt{2}\alpha y} \dd y + \int_1^\infty e^{\sqrt{2}\alpha y} e^{-y^2/s} \dd y \leq \frac{e^{\sqrt{2}\alpha}}{-\sqrt{2}\alpha} + s \sqrt{\pi s} e^{\alpha^2 s},
\end{align*}
by \eqref{base}. As a result, for all $A > 0$, we have
\begin{equation}
  \label{eqn:firstBit}
  \int_0^A \int_\R e^{(1-\alpha^2)s+\sqrt{2}\alpha y}u(y,s)^2\dd s \dd y = \int_0^A  e^{(1-\alpha^2)s} I_s \dd s < \infty.
\end{equation}

To complete the proof of \eqref{Phifini}, it is enough to bound $\int_A^\infty\int_\R e^{(1-\alpha^2)s+\sqrt{2}\alpha y}u(y,s)^2\dd y \dd s$ for $A\geq1$ large enough. Recall that $u(y,s) = \P(M_s \leq y)$ is close to $1$ for $y \gg \sqrt{2}s$ and to $0$ for $y \ll \sqrt{2}s$. Observe that
\begin{align}
  \int_A^\infty\int_{\sqrt{2}s}^\infty e^{(1-\alpha^2)s+\sqrt{2}\alpha y}u(y,s)^2\dd y \dd s
  \leq & \int_A^\infty\int_{\sqrt{2}s}^\infty e^{(1-\alpha^2)s+\sqrt{2}\alpha y} \dd y \dd s \nonumber \\
  = & \frac{1}{-\sqrt{2}\alpha} \int_A^\infty e^{(1-\alpha^2)s + 2\alpha s} \dd s < \infty, \label{eqn:first}
\end{align}
using that for all $\alpha < -\gamma$, $1 -\alpha^2 + 2\alpha < 0$. Therefore, we only need to bound
\begin{equation*}
 \int_A^\infty\int_{-\infty}^{\sqrt{2}s}e^{(1-\alpha^2)s+\sqrt{2}\alpha y}u(y,s)^2\dd y \dd s
  =\int_A^\infty\int_{-\infty}^{1}e^{(1-\alpha^2)s+2x\alpha s}u(\sqrt{2}xs,s)^2\sqrt{2}s\dd x \dd s,
\end{equation*}
by change of variable $y=\sqrt{2}sx$. We now apply Lemma~\ref{DSbdu} to bound $u(\sqrt{2}xs,s)^2$ for $s$ large enough, depending on the region to which $x$ belongs.

Let $\epsilon > 0$, that will be taken small enough later on, and $\beta  > 1 -\alpha /2 > 1$. We assume that $A > t_{\epsilon,\beta}$, and we bound the above integral using \eqref{DSupbdu}. First, for $x$ in the interval $[-\gamma,1]$, we have
\begin{align*}
  &\int_A^\infty \int_{-\gamma}^{1}e^{(1-\alpha^2)s+2\alpha xs} u(\sqrt{2}x s,s)^2 \sqrt{2}s \dd x \dd s\\
  \leq &\int_A^\infty \int_{-\gamma}^{1}e^{(1-\alpha^2)s+2\alpha xs} e^{-4 \gamma (1-x)s + 2\epsilon s} \sqrt{2}s \dd x \dd s \\
  \leq &\int_A^\infty \sqrt{2}s e^{(1-\alpha^2-4 \gamma + 2 \epsilon) s} \int_{-\gamma}^{1} e^{(2\alpha + 4 \gamma)xs} \dd x \dd s\\
  \leq &
    \begin{cases}
      \frac{1}{\sqrt{2}(\alpha+2\gamma)} \int_A^\infty e^{(1-\alpha^2+2\alpha)s+2\epsilon s}\dd s, &\textrm{ if }-2\gamma<\alpha<-\gamma;\\
      \sqrt{2}(1+\gamma)\int_A^\infty se^{(1-\alpha^2+2\alpha)s+2\epsilon s}\dd s, &\textrm{ if }\alpha=-2\gamma;\\
      \frac{1}{-\sqrt{2}(\alpha+2\gamma)} \int_A^\infty e^{(1-\alpha^2-4\gamma-2\alpha\gamma-4\gamma^2)s+2\epsilon s}\dd s, &\textrm{ if }\alpha<-2\gamma.
    \end{cases}
\end{align*}
As $1-\alpha^2-4\gamma-2\alpha\gamma-4\gamma^2<0$ for  all $\alpha < -2\gamma$, we conclude that for all $\epsilon>0$ small enough,
\begin{equation}
  \label{eqn:second}
  \int_A^\infty \int_{-\gamma}^{1}e^{(1-\alpha^2)s+2\alpha xs} u(\sqrt{2}x s,s)^2 \sqrt{2}s \dd x \dd s < \infty.
\end{equation}
We then consider the case $x \in [-\beta,-\gamma]$. In fact
\begin{align}
  &\int_A^\infty \int_{-\beta}^{-\gamma}e^{(1-\alpha^2)s+2\alpha xs} u(\sqrt{2}x s,s)^2 \sqrt{2}s \dd x \dd s\nonumber\\
  \leq & \int_A^\infty \int_{-\beta}^{-\gamma}e^{(1-\alpha^2)s+2\alpha xs} e^{-2(1 + x^2)s + 2 \epsilon s} \sqrt{2}s \dd x \dd s\nonumber\\
  \leq & \int_A^\infty  \sqrt{2}s e^{-(1+\alpha^2/2 - 2 \epsilon)s} \int_{-\beta}^{-\gamma} e^{-2(x -\alpha/2)^2 s} \dd x \dd s\nonumber\\
  \leq &\int_A^\infty  \sqrt{\frac{\pi s}{2}} e^{-(1+\alpha^2/2 - 2 \epsilon)s} \dd s  < \infty, \label{eqn:third}
\end{align}
for all $\epsilon<1/2$. Similarly, for $x < -\beta$:
\begin{align*}
  \int_A^\infty \int_{-\infty}^{-\beta}e^{(1-\alpha^2)s+2\alpha xs} u(\sqrt{2}x s,s)^2 \sqrt{2}s \dd x \dd s
  \leq &\int_A^\infty \sqrt{2} s e^{(1 -\alpha^2/2) s} \int_{-\infty}^{-\beta} e^{-2 (x-\alpha/2)^2s} \dd x \dd s\\
  \leq &\int_A^\infty \sqrt{2} s e^{(1 -\alpha^2/2) s} \int_{-\infty}^{-\beta-\alpha/2} e^{-2 y^2s} \dd y \dd s.
\end{align*}
Using that $- \beta -\alpha/2 < -1$, we have for all $s > 0$:
\[
  \int_{-\infty}^{-\beta-\alpha/2} e^{-2 y^2s} \dd y \leq \int_{1}^{\infty} e^{-2y^2 s} \dd y \leq \frac{1}{4 s} e^{-2s},
\]
by \eqref{basic}, yielding
\begin{equation}
  \label{eqn:fourth}
  \int_A^\infty \int_{-\beta}^{-\gamma}e^{(1-\alpha^2)s+2\alpha xs} u(\sqrt{2}x s,s)^2 \sqrt{2}s \dd x \dd s \leq \frac{1}{2\sqrt{2}} \int_A^\infty e^{(-1 -\alpha^2/2) s} < \infty.
\end{equation}
Consequently, using (\ref{eqn:first}--\ref{eqn:fourth}), for any $A>0$ large enough
\[
  \int_A^\infty\int_{-\infty}^{1}e^{(1-\alpha^2)s+2x\alpha s}u^2(\sqrt{2}xs,s)\sqrt{2}s\dd x \dd s<\infty,
\]
which, with \eqref{eqn:firstBit}, completes the proof of \eqref{Phifini}.
\end{proof}

We now complete the proof of Theorem~\ref{lowmax} by proving the joint convergence in law of the first branching time and position, and the shifted extremal process, conditioned on $\{M_t \leq \sqrt{2}\alpha t\}$.
\begin{proof}[Proof of Theorem~\ref{lowmax}]
We begin by observing that by Lemma \ref{lemcvgI1}, we have
\begin{multline}
  \E\left( e^{-\int \phi \dd \mathcal{E}_t(\alpha)} ; \tau \geq t, M_t\leq \sqrt{2}\alpha t-z \right)\\
  =\E\left( e^{-\phi(B_t-\sqrt{2}\alpha t)} ; B_t \leq \sqrt{2}\alpha t - z \right) \sim \frac{e^{-(1+\alpha^2)t}}{\sqrt{2 \pi t}} \int_{-\infty}^{-z} e^{-\phi(y) - \sqrt{2}\alpha y} \dd y, \label{eqn:noBranchingLow}
\end{multline}
as $t \to \infty$. Similarly to the proof of Theorem~\ref{highmax}, the key to the proof of this theorem is the determination of the asymptotic behaviour of
\begin{multline*}
  F_t(\phi;x_1,x_2,x_3)\\
  :=  \E\left( e^{-\int \phi \dd \mathcal{E}_t(\alpha)} ; \tau \in [ t - x_1,t], X_\emptyset(\tau \wedge t) \leq \sqrt{2}\alpha t  - x_2, M_t\leq \sqrt{2}\alpha t-x_3 \right),
\end{multline*}
as $t \to \infty$.

Using the branching property at time $\tau$, we observe that
\[
  F_t(\phi;x_1,x_2,x_3) = U_2^{\tau_{x_3} \phi} (\sqrt{2}\alpha t, t, [0,x_1],[x_2,\infty) ),
\]
and therefore, by Lemma~\ref{smalltime} we have
\begin{equation}
  \label{eqn:masterLow}
  \lim_{t \to \infty} t^{1/2} e^{(1+\alpha^2)t} F_t(\phi;x_1,x_2,x_3) = \int_0^{x_1} e^{(1-\alpha^2)s}\int_{x_2}^{\infty} e^{\sqrt{2}\alpha y} u_\phi(y - x_3,s)^2 \dd y \dd s.
\end{equation}
We now use this formula to prove Theorem \ref{lowmax}. 

We begin with the proof of \eqref{lowerprob-}. Observe that by \eqref{key+}, we have
\[
  \P(M_t \leq \sqrt{2}\alpha t) = U_1(\sqrt{2}\alpha t, t)+U_2(\sqrt{2}\alpha t, t).
\]
Using \eqref{eqn:masterLow} with $x_2 = -\infty$ and $x_1 = A$ together with Lemma \ref{largetime}, we have letting $t \to \infty$ then $A \to \infty$:
\[
  \lim_{t \to \infty} t^{1/2} e^{(1+\alpha^2)t}  U_2(\sqrt{2}\alpha t, t) = \int_0^{\infty}e^{(1-\alpha^2)s} \int_{\R} e^{\sqrt{2}\alpha y} u(y - x_3,s)^2 \dd y \dd s,
\]
which, together with \eqref{eqn:noBranchingLow}, implies $ \P(M_t \leq \sqrt{2}\alpha t) \sim \frac{\Phi(\alpha)}{\sqrt{4\pi t}} e^{-(1+\alpha^2)t}$ as $t \to \infty$.

Then, to prove \eqref{jointcvg-}, it is enough to observe that
\[
  \lim_{t \to \infty} \frac{F_t(0;x_1,x_2,x_3)}{\P(M_t \leq \sqrt{2}\alpha t)} = \frac{\sqrt{4\pi}}{\Phi(\alpha)} \int_0^{x_1} e^{(1-\alpha^2)s} e^{\sqrt{2}\alpha x_3} \int_{-\infty}^{x_2-x_3} e^{\sqrt{2}\alpha y} u_\phi(y,s)^2 \dd y \dd s,
\]
by \eqref{eqn:masterLow}. This proves that $(t - t\wedge \tau, \sqrt{2}\alpha t - X_\emptyset(t\wedge \tau), \sqrt{2}\alpha t - M_t)$ jointly converge in distribution as $t \to \infty$.

We now prove the convergence of the extremal process $\mathcal{E}_t(\alpha)$. For any $\phi\in \mathcal{C}_c^+(\R)$, using again the decomposition at first branching time of the branching Brownian motion, we have
\[
  \E\left[e^{-\int \phi \dd\mathcal{E}_t(\alpha)}; M_t\leq\sqrt{2}\alpha t\right] = U^\phi_1(\sqrt{2}\alpha t, t) + U^\phi_2(\sqrt{2}\alpha t, t),
\]
which, by \eqref{eqn:masterLow} and \eqref{lowerprob-}, yields
\begin{multline*}
  \lim_{t \to \infty} \E\left[e^{-\int \phi \dd\mathcal{E}_t(\alpha)}\middle| M_t\leq\sqrt{2}\alpha t\right]\\
   = \frac{\sqrt{2}}{\Phi(\alpha)} \left( \int_{-\infty}^0 e^{-\phi(z)-\sqrt{2}\alpha z}\dd z +\int_0^\infty\dd s\int_\R e^{\sqrt{2}\alpha z+(1-\alpha^2)s}u_\phi(z,s)^2\dd z\right).
\end{multline*}
As $u(z,s) = \E\left( e^{-\sum_{u \in N(t)} \phi(X_u(s)-z)}; M_s \leq z \right)$, we observe that we can rewrite this limit as
\begin{multline*}
  \int_{\R_+\times\R}\E\left[\exp\left(-\sum_{u\in N(s)}\phi(X_u(s)-z)\right)\middle\vert M_{s}\leq z \right]^2\P(\xi_\alpha\in\dd s, \chi_\alpha\in\dd z)\\
  =\E[e^{-\int \phi \dd \mathcal{E}_\infty(\alpha)}],
\end{multline*}
proving that $\mathcal{E}_t(\alpha)$ converges weakly to $\mathcal{E}_\infty(\alpha)$ in $\P(\cdot|M_t \leq \sqrt{2}\alpha t)$-distribution. In the same spirit as in the proof of Theorem~\ref{highmax}, one could obtain the joint convergence in distribution of $\mathcal{E}_t(\alpha)$ with $\left(t-t\wedge \tau, X_\emptyset(t\wedge \tau)\right)$, thus completing the proof of \eqref{processcvg-}.
\end{proof}

\section{The critical case \texorpdfstring{$\alpha=-\gamma$}{alpha equal to minus gamma}: Proof of Theorem \ref{criticalmax}}
\label{maxc}

We consider in this section the case $\alpha = -\gamma$ and prove Theorem \ref{criticalmax}. According to this theorem, the first branching time should occur around time $t - O(t^{1/2})$ with high probability. We use again \eqref{key} to compute the asymptotic behaviour of $\P(M_t \leq -\sqrt{2}\gamma t)$, and decompose the integral \eqref{eqn:keyPhi} onto sub-intervals of interest to prove the joint convergence in distribution of the first branching time and position and the extremal process of the branching Brownian motion.

Recall that, by \eqref{key+}, we have $u(-\sqrt{2}\gamma t, t)=U_1(-\sqrt{2}\gamma t, t)+U_2(-\sqrt{2}\gamma t,t)$, and by Lemma \ref{lemcvgI1}, we have, as $t \to \infty$,
\begin{equation}
  \label{eqn:oneObs}
  U_1(-\sqrt{2}\gamma t,t) \sim \frac{1}{\sqrt{4\pi t}}e^{-(1+\gamma^2)t} \ll t^{3\gamma/4}e^{-(1+\gamma^2)t}.
\end{equation}
Therefore, to complete the proof of \eqref{lowerprobc}, it is enough to show that
\begin{equation}\label{1birthc}
  U_2(-\sqrt{2}\gamma t, t)\sim C^{(2)} t^{3\gamma/4}e^{-(1+\gamma^2)t}, \text{ as } t \to \infty.
\end{equation}
Moreover, for any $0 < a < b < t$ and $c < d$, we set
\[
U_2(-\sqrt{2}\gamma t, t, [a,b],[c,d]):=\int_{a}^{b} e^{-(t-s)}ds\int_c^d \frac{e^{-\frac{(z+m_s+\sqrt{2}\gamma t-a_t)^2}{2(t-s)}}}{\sqrt{2\pi(t-s)}}u(m_s+z,s)^2\dd z.
\]
Equation \eqref{1birthc} follows from the next two lemmas.

\begin{lemma}
For all $A, K>0$, we have
\label{badtimecA}
\[
  \limsup_{t\to\infty} \frac{e^{(1+\gamma^2)t}}{t^{3\gamma/4}} \left|U_2(-\sqrt{2}\gamma t, t)-U_2(-\sqrt{2}\gamma t, t, [\sqrt{t}/A, A\sqrt{t}],[-K,K])\right| = o_{A}(1)+o_{K}(1).
\]
\end{lemma}

This lemma allows to localise the first branching time and position, conditioned on $\{M_t \leq -\sqrt{2}\gamma t\}$. We note that it is similar to the proof of Lemma~\ref{badtime+badposition}, and postpone its proof to Appendix~\ref{lems4}. The next lemma gives a more detailed estimate of the time at which this branching event occurs.

\begin{lemma}\label{goodtimec}
For any $0 < a < b$ and $c<d$ fixed, for all $\phi \in \mathcal{C}_c^+(\R)$, we have
\begin{multline*}
  \lim_{t\to\infty}\frac{e^{(1+\gamma^2)t}}{t^{3\gamma/4}}U^\phi_2(-\sqrt{2}\gamma t, t, [a\sqrt{t}, b\sqrt{t}],[c,d])\\
 =\frac{1}{\sqrt{2\pi}}\int_{a}^b r^{3\gamma/2}e^{-2r^2}\dd r \int_{c}^d e^{-\sqrt{2}\gamma z}w_\phi(z)^2\dd z.
\end{multline*}
\end{lemma}

\begin{proof}
This proof is similar to the proofs of Lemmas \ref{goodtime+position} and \ref{smalltime}. By \eqref{uniformcvg}, we have
\begin{align*}
  U^\phi_2(-\sqrt{2}\gamma t, t, [a\sqrt{t}, b\sqrt{t}],[c,d])
  =&\int_{a\sqrt{t}}^{b\sqrt{t}}e^{-(t-s)}\dd s\int_{c}^d\frac{e^{-\frac{(z+m_s+\sqrt{2}\gamma t)^2}{2(t-s)}}}{\sqrt{2\pi(t-s)}}u_\phi(m_s+z,s)^2\dd z\\
  \sim &\int_{a\sqrt{t}}^{b\sqrt{t}}e^{-(t-s)}\dd s\int_{c}^d\frac{e^{-\frac{(z+m_s+\sqrt{2}\gamma t)^2}{2(t-s)}}}{\sqrt{2\pi(t-s)}}w_\phi(z)^2\dd z,
\end{align*}
as $t \to \infty$. Note that $1+\gamma^2=2\sqrt{2}\gamma$. Hence by simple calculations we obtain
\begin{multline*}
  U^\phi_2(-\sqrt{2}\gamma t, t, [a\sqrt{t}, b\sqrt{t}],[c,d])\\
  \sim e^{-(1+\gamma^2)t}\int_{a\sqrt{t}}^{b\sqrt{t}}s^{3\gamma/2}e^{-\frac{2s^2}{t}}\frac{\dd s}{\sqrt{2\pi t}}\int_{c}^d e^{-\sqrt{2}\gamma z}w_\phi^2(z)\dd z, \quad \text{as $t \to \infty$.}
\end{multline*}
With a change of variable $s=r\sqrt{t}$, the proof is now complete.
\end{proof}

\begin{proof}[Proof of \eqref{lowerprobc}]
Recall that it is enough to prove \eqref{1birthc}. For all $t > 0$ and $A, K > 0$, we have
\begin{multline*}
  U_2(-\sqrt{2}\gamma t, t) = \left(U_2(-\sqrt{2}\gamma t, t)-U_2(-\sqrt{2}\gamma t, t, [\sqrt{t}/A, A\sqrt{t}],[-K,K])\right)\\ + U_2(-\sqrt{2}\gamma t, t, [\sqrt{t}/A, A\sqrt{t}],[-K,K]).
\end{multline*}
Therefore, using Lemmas \ref{badtimecA} and \ref{goodtimec}, we obtain
\begin{multline*}
  \lim_{t \to \infty}  t^{-3\gamma/4}e^{(1+\gamma^2)t} U_2(-\sqrt{2}\gamma t, t)\\
  = o_A(1) + o_K(1) + \frac{1}{\sqrt{2\pi}}\int_{1/A}^A r^{3\gamma/2}e^{-2r^2}\dd r \int_{-K}^K e^{-\sqrt{2}\gamma z}w^2(z)\dd z,
\end{multline*}
which converges to $C^{(2)}$ as $A,K \to \infty$.
\end{proof}

We now complete the proof of Theorem \ref{criticalmax} by proving \eqref{jointcvgc} and \eqref{processcvgc}.
\begin{proof}[Proof of \eqref{jointcvgc} and \eqref{processcvgc}]
For any $x_1,x_3\in\R_+$ and $x_2\in\R$, applying the Markov property at time $\tau$, and using \eqref{eqn:oneObs}, we have
\begin{align*}
&\P\left(\tau \geq t-x_1\sqrt{t}, X_\emptyset(\tau)\geq (-\sqrt{2}\gamma t+m_{t-\tau})-x_2, M_t\leq -\sqrt{2}\gamma t-x_3\right)\\
=& \int_{t-x_1\sqrt{t}}^t e^{-r}\dd r\int_{(-\sqrt{2}\gamma t-m_{t-r})-x_2}^\infty\P(B_r \in\dd y)u^2(-\sqrt{2}\gamma t-x_3-y, t-r)\\
&\qquad\qquad\qquad\qquad\qquad\qquad\qquad\qquad\qquad\qquad\qquad\qquad\qquad + o_t(1)t^{3\gamma/4}e^{-(1+\gamma^2)t}\\
=&\int_0^{x_1\sqrt{t}}\dd s\int_{-\infty}^{x_2-x_3}\frac{e^{-\frac{(z+m_s+\sqrt{2}\gamma t+x_3)^2}{2(t-s)}}}{\sqrt{2 \pi (t-s)}}u^2(m_s+z,s)\dd z + o_t(1)t^{3\gamma/4}e^{-(1+\gamma^2)t},
\end{align*}
using the change of variables $s = t-r$ and $z = -\sqrt{2}\gamma t - x_3-y-m_s$. We now apply Lemma \ref{badtimecA} to obtain
\begin{align*}
  &\P\left(\tau \geq t-x_1\sqrt{t}, X_\emptyset(\tau)\geq (-\sqrt{2}\gamma t+m_{t-\tau})-x_2, M_t\leq -\sqrt{2}\gamma t-x_3\right)\\
  =&(o_{A,t}(1)+o_{K,t}(1)+o_t(1))t^{3\gamma/4}e^{-(1+\gamma^2)t}\\
  &\qquad \qquad +\int_{\sqrt{t}/A}^{x_1\sqrt{t}}\dd s\int_{-K}^{x_2-x_3}\frac{e^{-\frac{(z+m_s+\sqrt{2}\gamma t+x_3)^2}{2(t-s)}}}{\sqrt{2\pi (t-s)}}u^2(m_s+z,s)\dd z.
\end{align*}
Then Lemma \ref{goodtimec} gives
\begin{multline*}
  \P\left(\tau \geq t-x_1\sqrt{t}, X_\emptyset(\tau)\geq (-\sqrt{2}\gamma t+m_{t-\tau})-x_2, M_t\leq -\sqrt{2}\gamma t-x_3\right)\\
\sim \frac{1}{\sqrt{2\pi}}t^{3\gamma/4}e^{-(1+\gamma^2)t}\int_0^{x_1} r^{3\gamma/2}e^{-2r^2}\dd r\int_{-\infty}^{x_2-x_3}e^{-\sqrt{2}\gamma(z+x_3)}w^2(z)\dd z \text{ as } t \to \infty,
\end{multline*}
which, together with \eqref{lowerprobc} proves \eqref{jointcvgc}.

We now turn to the proof of \eqref{processcvgc}. For any $\phi\in \mathcal{C}_c^+(\R)$, using again the Markov property at time $\tau$, we have
\begin{align*}
&\E[e^{-\int\phi \dd \mathcal{E}_t(-\gamma)}; M_t\leq -\sqrt{2}\gamma t]\\
=& e^{-t} \E[e^{-\phi(B_t+\sqrt{2}\gamma t)}; B_t\leq -\sqrt{2}\gamma t]\\
&+\int_0^t e^{-r}\dd r\int_\R \P(B_r\in\dd y)\E\left[e^{-\sum_{u\in N(t-r)}\phi(X_u(t-r)+y+\sqrt{2}\gamma t)}; M_{t-r}\leq -\sqrt{2}\gamma t-y\right]^2.
\end{align*}
On the one hand, by \eqref{cvgI1},
\[
  e^{-t} \E[e^{-\phi(B_t+\sqrt{2}\gamma t)}; B_t\leq -\sqrt{2}\gamma t]\leq e^{-t}\P(B_t\leq -\sqrt{2}\gamma t)=o_t(1)t^{3\gamma/4}e^{-(1+\alpha^2)t}.
\]
On the other hand, using again Lemmas \ref{badtimecA} and \ref{goodtimec}, we obtain
\begin{align*}
  &\int_0^t e^{-r}\dd r\int_\R \P(B_r\in\dd y)\E[e^{-\sum_{u\in N(t-r)}\phi(X_u(t-r)+y+\sqrt{2}\gamma t)}; M_{t-r}\leq -\sqrt{2}\gamma t-y]^2 \\
  \sim &t^{3\gamma/4}e^{-(1+\gamma^2)t}\int_{0}^\infty r^{3\gamma/2}e^{-2r^2}\frac{dr}{\sqrt{2\pi }}\int_{\R} e^{-\sqrt{2}\gamma z}w_\phi^2(z)\dd z,
\end{align*}
as $t \to \infty$. It thus follows, using \eqref{lowerprobc}, that
\begin{align*}
 \lim_{t \to \infty} \E\left[e^{-\int\phi d\mathcal{E}_t(-\gamma)}\vert M_t\leq -\sqrt{2}\gamma t\right]=
\frac{1}{2C^{(1)}}\int_\R e^{-\sqrt{2}\gamma z}w_\phi^2(z)\dd z=\E\left[e^{-\int\phi d\mathcal{E}^-}\right],
\end{align*}
which, by \cite[Proposition 2.2]{Bo17} is enough to conclude \eqref{processcvgc}.
\end{proof}

\appendix

\section{Proof of Lemmas}
\label{lems}

We prove in this section some of the more technical lemmas, that are needed to complete the proofs.

\subsection{Proof of Lemma \ref{CHbdu}}
\label{lems1}

Recall that Lemma \ref{CHbdu} consists in the following non-asymptotic estimate :  for all $\delta > 0$, $\P(M_t\leq m_t-z)\leq c_\delta e^{-\sqrt{2}\gamma(1-\delta)z}$ for all $t,z \geq 1$.

\begin{proof}[Proof of Lemma \ref{CHbdu}]
We begin by bounding $\P(M_t \leq m_t - z)$ for $1 < z \leq t$. Denote by $n(t):=\#N(t)$ the total number of particles alive at time $t$. As every individual gives birth at exponential rate to two children, the process $(n(t), t \geq 0)$ is a standard Yule process. Hence $\P(n(t)=k)=e^{-t}(1-e^{-t})^{k-1}$ for any $k\in\N$. Let $\eta\in(0,\delta)$ small enough, such that $J:=\lfloor \frac{\sqrt{2}\gamma}{\eta}\rfloor\geq1$.  Observe that
\begin{align}\label{GHbd}
  \P(M_t\leq m_t-z)
  \leq &\P(n(Jz\eta )\leq z^3)+\P(n(Jz\eta)> z^3; M_t\leq m_t-z)\nonumber \\
  \leq & z^3e^{-Jz\eta}+\sum_{k=1}^J\P(n((k-1)z\eta)\leq z^3<n(kz\eta); M_t\leq m_t-z).
\end{align}

Let $(W_t)_{t\ge0}$ be a standard Brownian motion, independent of the branching Brownian motion. Using \cite[Lemma 5.1]{GH18}, for any $0<s<t$ and $x \in \R$, we have
\[
  \P(M_t \leq x) \leq \P\left(\max_{u\in N(s)}(W_{s}+M^u_{t-s}) \leq x\right),
\]
where $M^u_y:=\max_{v\in N(y+s), u\preccurlyeq v}X_v(y+s)-X_u(s)$, and $u\preccurlyeq v$ means that $v$ is a descendant of $u$. For any $1\leq k\leq J$, one has
\begin{align}
  \label{roughgoodtime}
  &\P(n((k-1)z\eta)\leq z^3<n(kz\eta); M_t\leq m_t-z)\nonumber\\
  \leq &\P(n((k-1)z\eta)\leq z^3<n(kz\eta); \max_{u\in N(kz\eta)}(W_{kz\eta}+M^u_{t-kz\eta})\leq m_t-z)\nonumber\\
  \leq &\P(n((k-1)z\eta)\leq z^3<n(kz\eta); W_{kz\eta}\leq m_t-m_{t-kz\eta}-z)+\P(M_{t-kz\eta}\leq m_{t-kz\eta})^{z^3}.
\end{align}

On the one hand, for $1\leq k\leq J$, by \eqref{basic},
\begin{align*}
  &\P(n((k-1)z\eta)\leq z^3<n(kz\eta); W_{kz\eta}\leq m_t-m_{t-kz\eta}-z)\\
  \leq&\P(n((k-1)z\eta)\leq z^3)\P(W_{kz\eta}\leq -(1-\sqrt{2}k\eta)z)\\
  \leq&z^3e^{-(k-1)z\eta}\frac{\sqrt{kz\eta}}{(1-\sqrt{2}k\eta)z}e^{-\frac{(1-\sqrt{2}k\eta)^2}{2k\eta}z}.
\end{align*}
As $(1-\sqrt{2}k\eta)\sqrt{z}\geq (1-2\gamma)\sqrt{z}>\sqrt{J\eta}$ for $z>100$, we deduce that
\begin{multline}
  \label{roughgoodtime1}
  \P(n((k-1)z\eta)\leq z^3<n(kz\eta); W_{kz\eta}
  \leq m_t-m_{t-kz\eta}-z)\\
  \quad\leq z^3e^{\eta z} \exp\left\{-\left[k\eta+\frac{(1-\sqrt{2}k\eta)^2}{2k\eta}\right]z\right\}.
\end{multline}

On the other hand, as $t-J\eta z\geq (1-\sqrt{2}\gamma)t\rightarrow\infty$ as $t\to\infty$, using that $z \leq t$ and the convergence \eqref{eqn:cvPointwise}, there exists $t_0\geq1$ and $c_0>0$ such that for all $t\geq t_0$ and $1\leq z\leq t$, one has $\displaystyle \P(M_{t-kz\eta} \leq m_{t - k z \eta}) \leq e^{-c_0}<1$. Then,
\begin{equation}\label{roughgoodtime2}
  \P(M_{t-kz\eta}\leq m_{t-kz\eta})^{z^3}\leq e^{-c_0 z^3}.
\end{equation}
As a result, using \eqref{roughgoodtime}, \eqref{roughgoodtime1} and \eqref{roughgoodtime2}, for $t>t_0$ and $100< z\leq t$, \eqref{GHbd} becomes that
\begin{align*}
  &\P(M_t\leq m_t-z)\\
  \leq & z^{3}e^{\eta z}e^{-\sqrt{2}\gamma z}+J\sup_{1\leq k\leq J}\left(z^3e^{\eta z} \exp\left(-\left[k\eta+\frac{(1-\sqrt{2}k\eta)^2}{2k\eta}\right]z\right)+e^{-c_0 z^3}\right)\\
  \leq & z^3e^{\eta z}e^{-\sqrt{2}\gamma z}+\frac{\sqrt{2}\gamma}{\eta}\sup_{0<s<\sqrt{2}\gamma}\left(z^3e^{\eta z} \exp\left(-\left[s+\frac{(1-\sqrt{2}s)^2}{2s}\right]z\right)+e^{-c_0 z^3}\right)\\
  =&z^3e^{\eta z}e^{-\sqrt{2}\gamma z}+\frac{\sqrt{2}\gamma}{\eta}\left(z^3e^{\eta z}e^{-\sqrt{2}\gamma z}+e^{-c_0 z^3}\right).
\end{align*}
For $\delta\in(0,1)$ small enough, we could take $\eta=\sqrt{2}\gamma\delta/2$, $t\geq t_0$ and $z\in[K_\delta, t]$ such that
\[
\P(M_t\leq m_t-z)\leq c_\delta e^{-\sqrt{2}\gamma (1-\delta) z}.
\]
Up to enlarging the constant $c_\delta$, this equation will hold for all $1 \leq z \leq t$.

We now bound $\P(M_t\leq m_t-z)$ with $z\geq t$. We apply \eqref{DSupbdu} and obtain that for $z\geq t\geq t_{\epsilon,\beta}$,
\begin{align*}
  \P(M_t\leq m_t-z)&\leq u\left(\sqrt{2}\left(1-\tfrac{z}{\sqrt{2}t}\right)t,t\right)\\
  & \leq
  \begin{cases}
    e^{-\sqrt{2}\gamma z+\epsilon t},&\textrm{ if } t\le z<2t;\\
    e^{-(1+(1-\frac{z}{\sqrt{2}t})^2)t+\epsilon t},&\textrm{ if } 2t\le z\leq \sqrt{2}(1+\beta)t;\\
    e^{-(1-\frac{z}{\sqrt{2}t})^2t},&\textrm{ if } z\geq \sqrt{2}(1+\beta)t.
  \end{cases}
\end{align*}
Note that $1+a^2\geq 2\gamma(1-a)$ for $a<1-\sqrt{2}$. So, $(1+(1-\frac{z}{\sqrt{2}t})^2)t\geq \sqrt{2}\gamma z$ if $z\ge 2t$. We also have $\left(1-\frac{z}{\sqrt{2}t}\right)^2t\geq \sqrt{2}\gamma z$ if $z\ge \sqrt{2}(1+\sqrt{2})t$. By taking $\beta=\sqrt{2}$ and $\epsilon=\sqrt{2}\gamma\delta$, we thus get that for $t\geq t_{\epsilon,\beta} $ and $z\ge t$,
\[
\P(M_t\leq m_t-z)\leq e^{-\sqrt{2}\gamma z+\epsilon t}\leq e^{-\sqrt{2}\gamma(1-\delta)z}.
\]
We hence conclude that for any $\delta\in(0,1)$, there exists $T_\delta=t_{\epsilon,\beta}\vee t_0$ and $K_\delta\ge1$ such that for any $t\geq t_\delta$ and $z\geq K_\delta$, \eqref{CHupbdu} holds. Thus, up to enlarging again constant $c_\delta$, the proof is now complete.
\end{proof}

\subsection{Proof of Lemma \ref{badtime+badposition}}
\label{lems2}

We assume here that $\alpha \in (-\gamma,1)$. The aim of this section is to prove that for all $\phi \in \mathcal{C}^+_c(\R)$, setting $I_{t,A} = [v_\alpha t - At^{1/2},v_\alpha t + A t^{1/2}]$, we have
\begin{equation}
  \label{eqn:data}
  \lim_{A,K \to \infty}
  \limsup_{t \to \infty} \frac{e^{2 \gamma(1-\alpha)t}}{t^{3 \gamma /2 }} \int_{(I_{t,A} \times [-K,K])^c}\!\!\!\!\!\!\!\!\dd s\dd z \frac{e^{-(t-s)-\frac{(m_s+z-\sqrt{2}\alpha t)^2}{2(t-s)}}}{\sqrt{2\pi (t-s)}}u_\phi(m_s+z,s)^2 = 0.
\end{equation}
As $\phi$ is non-negative, we observe that
\[
  u_\phi(z,t) = \E\left(e^{-\sum_{u \in N(t)} \phi(X_u(t)-z)};M_t \leq z\right) \leq \P(M_t \leq z) = u(z,t).
\]
It is enough to prove that \eqref{eqn:data} holds for $\phi \equiv 0$.

Therefore, the objective of the section can be restated as follows: conditioned on $\{M_t \leq \sqrt{2}\alpha t\}$, we show that the first branching time $\tau$ is with high probability located around $(1 - v_\alpha) t + O(\sqrt{t})$, and the position at which that particle branches satisfies $\sqrt{2}\alpha t - m_{t-\tau} + O(1)$ with high probability.

The idea of the proof is the following: we use \eqref{key} to rewrite $u$ as the sum of $U_1$ and $U_2$. By Lemma \ref{lemcvgI1}, $U_1$ add a negligible contribution to $u$, so that
\[
  u(\sqrt{2}\alpha t,t) \approx \int_{[0,t] \times \R} \dd s\dd z \frac{e^{-(t-s)-\frac{(m_s+z-\sqrt{2}\alpha t)^2}{2(t-s)}}}{\sqrt{2\pi (t-s)}}u(m_s+z,s)^2.
\]
Moreover, by Lemma \ref{goodtime+position}, a large contribution to $u$ is carried by the regions of the form $I_{A,t} \times [-K,K]$, with $A > 0$ and $K$ large enough. We now use a priori domination estimates for $u$ (e.g. Lemma \ref{DSbdu}) and methods similar to the proof of Laplace's method.

We decompose the proof of Lemma \ref{badtime+badposition} into three parts, by considering the contribution of various domains of $[0,t] \times \R$.

\subsubsection{Linear bounds on the first splitting time}

As a first step towards the proof of Lemma \ref{badtime+badposition}, we show that for all $\epsilon > 0$,
\[
\P\left(|\tau - (1-v_\alpha) t| > \epsilon t , M_t \leq \sqrt{2}\alpha t\right) \ll u(\sqrt{2}\alpha t , t).
\]
\begin{lemma}
\label{small2ndterm}
Let $\alpha \in (-\gamma,1)$. For all $\epsilon > 0$ small enough, we have
\begin{align}
  \limsup_{t \to \infty}& \frac{1}{t} \log U_2(\sqrt{2}\alpha t, t,[0,(v_\alpha - \epsilon) t]) < - 2 \gamma(1-\alpha), \label{eqn:ubBeginning} \\
  \limsup_{t \to \infty}& \frac{1}{t} \log U_2(\sqrt{2}\alpha t, t,[(v_\alpha + \epsilon) t,t]) < - 2 \gamma(1-\alpha). \label{eqn:ubEnding}
\end{align}
\end{lemma}

To prove this result, we begin by bounding the probability that a split occurs at the very end of the process.
\begin{lemma}
\label{lem:upbdsmalltime}
Let $\alpha \in (-\gamma,1)$. There exists $\epsilon_0 > 0$ such that for all $\epsilon \in (0,\epsilon_0)$,
\begin{equation}
  \label{upbdsmalltime}
  \limsup_{t \to \infty} \frac{1}{t} \log U_2(\sqrt{2}\alpha t, t,[0,\epsilon t]) < -2\gamma(1-\alpha).
\end{equation}
\end{lemma}

\begin{proof}
Equation \eqref{upbdsmalltime} can be rewritten as
\[
  \P(0 \leq t - \tau \leq \epsilon t, M_t \leq \sqrt{2}\alpha t) \ll t^{3\gamma/2} e^{-2 \gamma(1-\alpha)t}.
\]
First note that
$\displaystyle \P(0 \leq t - \tau \leq \epsilon t, M_t \leq \sqrt{2}\alpha t) \leq \P(\tau \geq (1- \epsilon) t) =e^{-(1-\epsilon)t}$. Thus \eqref{upbdsmalltime} holds for all $\alpha$ such that $2 \gamma (1-\alpha) < 1$, i.e. $\alpha > -\gamma/2$, for all $\epsilon > 0$ small enough.

We now assume that $\alpha \leq -\gamma/2 < 0$, and decompose the above probability as
\begin{multline}
  \label{eqn:ppov}
  \P(0 \leq t - \tau \leq \epsilon t, M_t \leq \sqrt{2}\alpha t) \leq\\
   \P(0 \leq t - \tau \leq \epsilon t, X_\emptyset(\tau) \leq \sqrt{2}(\alpha + 2 \epsilon)t, M_t \leq \sqrt{2}\alpha t)\\
    + \P(0 \leq t - \tau \leq \epsilon t, X_\emptyset(\tau) \geq \sqrt{2}(\alpha + 2\epsilon) t, M_t \leq \sqrt{2}\alpha t),
\end{multline}
and bound these two quantities separately.

First note that
\begin{multline*}
  \P(0 \leq t - \tau \leq \epsilon t, X_\emptyset(\tau) \leq \sqrt{2}(\alpha + 2 \epsilon)t, M_t \leq \sqrt{2}\alpha t)\\
  \leq \P(0 \leq t - \tau \leq \epsilon t, X_\emptyset(\tau) \leq \sqrt{2}(\alpha + 2\epsilon)t)\qquad \qquad \qquad\\
  \leq \int_{(1-\epsilon)t}^t e^{-s} \P(B_s \leq \sqrt{2}(\alpha + 2 \epsilon)t) \dd s \leq C t^{-1/2}  e^{-(1-\epsilon)t} e^{-\frac{(\alpha + 2 \epsilon)^2}{2(1-\epsilon)} t},
\end{multline*}
using \eqref{basic}. As $1 + \frac{\alpha^2}{2} > 2 \gamma (1-\alpha)$ for all $\alpha \in (-\gamma,-\gamma/2]$, we deduce that for all $\epsilon > 0$ small enough, there exists $\delta > 0$ such that
\begin{equation}
  \label{eqn:part1}
  \P(0 \leq t - \tau \leq \epsilon t, X_\emptyset(\tau) \leq \sqrt{2}(\alpha  + 2\epsilon)t) \leq C e^{-2 \gamma(1-\alpha)t - \delta t}.
\end{equation}
We now turn to bounding the second probability in \eqref{eqn:ppov}

Using the Markov property at time $\tau$, we bound it as
\begin{multline*}
  \P(0 \leq t - \tau \leq \epsilon t, X_\emptyset(\tau) \geq \sqrt{2}(\alpha + 2 \epsilon) t, M_t \leq \sqrt{2}\alpha t)\\
  \leq \int_{(1-\epsilon)t}^t e^{-s} \E\left(u(t-s,\sqrt{2}\alpha t- B_s)^2 \ind{B_s \geq \sqrt{2}(\alpha +2\epsilon)t}\right)  \dd s.
\end{multline*}
By Lemma \ref{DSbdu}, for all $s < \epsilon t$ and $y \leq - 2\epsilon t$, we have $u(s,y) \leq e^{-y^2/2s}$, yielding
\begin{multline*}
  \P(0 \leq t - \tau \leq \epsilon t, X_\emptyset(\tau) \geq \sqrt{2}(\alpha + 2 \epsilon) t, M_t \leq \sqrt{2}\alpha t)\\
  \leq e^{-(1 - \epsilon) t} \int_{0}^{\epsilon t} \int_{\sqrt{2}(\alpha + 2\epsilon)t}^\infty  e^{-\frac{(\sqrt{2}\alpha t - y)^2}{s}} e^{-\frac{y^2}{2(t-s)}} \dd y \dd s.
\end{multline*}
Using that
\[
  -\frac{y^2}{2(t-s)}-\frac{(y-\sqrt{2}\alpha t)^2}{s}=-\frac{2\alpha^2t^2}{2t-s}-\frac{(y-\sqrt{2}\alpha \frac{2(t-s)}{2t-s})^2}{\frac{2(t-s)s}{2t-s}},
\]
we obtain
\begin{multline*}
  \P(0 \leq t - \tau \leq \epsilon t, X_\emptyset(\tau) \geq \sqrt{2}(\alpha + 2 \epsilon) t, M_t \leq \sqrt{2}\alpha t)\\
  \leq e^{-(1+\alpha^2 - \epsilon) t} \int_{0}^{\epsilon t} \int_{\sqrt{2}(\alpha + 2\epsilon - \frac{2(t-s)}{2t-s})t}^\infty  e^{- \frac{z^2}{\frac{2(t-s)s}{2t-s}}} \dd z \dd s
  \leq \frac{\sqrt{2\pi}\epsilon^{3/2}}{\sqrt{1-\epsilon/2}} t^{3/2} e^{-(1+\alpha^2 - \epsilon) t}.
\end{multline*}
Therefore, as $1 +\alpha^2 > 2 \gamma (1 -\alpha)$ for all $\alpha \in (-\gamma,1)$, we conclude that for all $\epsilon > 0$ small enough, there exists $\delta >0$ such that
\begin{equation}
  \label{eqn:part2}
  \P(0 \leq t - \tau \leq \epsilon t, X_\emptyset(\tau) \geq \sqrt{2}(\alpha + 2 \epsilon) t, M_t \leq \sqrt{2}\alpha t) \leq C e^{-2 \gamma(1-\alpha)t - \delta t}.
\end{equation}

In view of \eqref{eqn:ppov}, equations \eqref{eqn:part1} and \eqref{eqn:part2} show that there exists $\epsilon_0$ so that for all $0 < \epsilon < \epsilon_0$, \eqref{upbdsmalltime} holds.
\end{proof}

We now bound the probability that the first splitting time in the branching Brownian motion occurs after time at a distance at least $\epsilon t$ from the expected time $(1 - v_\alpha) t$.
\begin{lemma}
\label{lem:upbdmediatetime}
Let $\alpha \in (-\gamma,1)$. There exists $\epsilon_1 > 0$ such that for all $\epsilon \in (0,\epsilon_1)$,
\begin{align}
  \limsup_{t \to \infty} \frac{1}{t} \log U_2(\sqrt{2}\alpha t, t,[\epsilon t, (v_\alpha-\epsilon)t]) < -2 \gamma (1-\alpha), \label{upbdmediatime}\\
  \limsup_{t \to \infty} \frac{1}{t} \log U_2(\sqrt{2}\alpha t, t,[(v_\alpha+\epsilon)t,t])< -2\gamma (1-\alpha). \label{upbdlargetime+}
\end{align}
\end{lemma}

\begin{proof}
Let $a < b$ such that $[a,b] \subset (0,v_\alpha) \cup (v_\alpha,1]$. By definition of $U_2$, we have
\begin{align*}
  U_2(\sqrt{2}\alpha t, t,[at ,bt])
  \leq & \int_{a t}^{b t} \int_\R\frac{\dd z}{\sqrt{2 \pi t}} e^{-(t-s) - \frac{(z - \sqrt{2}\alpha t)^2}{2(t-s)}} u(z,s)^2 \dd z  \dd s\\
  \leq & \int_{a}^{b} \int_\R e^{-t(1-r) - t\frac{(\sqrt{2}h r - \sqrt{2}\alpha )^2}{2(1-r)}} u(\sqrt{2}h tr , tr)^2 \sqrt{2}t^{3/2} r \dd r \frac{\dd h}{\sqrt{2\pi}},
\end{align*}
by change of variables $r= s/t$ and $h = z/\sqrt{2}s$. We then use Lemma \ref{DSbdu} to bound $u(\sqrt{2}h tr, r)$ uniformly in $(h,r)$ for $t$ large enough. For all $\delta > 0$ and $\beta \geq 1$, for all $t$ large enough we have
\begin{equation}
  \label{eqn:theAim}
  U_2(\sqrt{2}\alpha t, t,[at,bt]) \leq \frac{t^{3/2}}{\sqrt{\pi}} \int_{a}^{b} \int_\R e^{- t\left( 1 - r + \frac{(hr-\alpha)^2}{1-r} + 2 \bar{\Psi}_\beta(h) - \delta\right)}  \dd h \dd r,
\end{equation}
where we set
\begin{equation}
  \label{eqn:definePsiBetaBar}
  \bar{\Psi}_\beta(a) :=
  \begin{cases}
    0, & \text{ if } a \geq 1;\\
    \sqrt{2} \gamma (1 - a), & \text{ if } - \gamma \leq a < 1;\\
    (1 + a^2), & \text{ if } -\beta \leq a < -\gamma;\\
    a^2, &\text{ if } a < -\beta.
  \end{cases}
\end{equation}
To complete this proof, it is therefore enough to prove that the right-hand side of \eqref{eqn:theAim} decays exponentially fast, at a rate larger than $2 \gamma(1-\alpha)$. To do so, we decompose the integral over $\R$ into thee subsets : $(-\infty,-\beta)$, $[-\beta,1]$ and $(1,\infty)$.

We first observe that on the interval $[1,\infty)$, by change of variable $v = hr-\alpha$, we have
\begin{multline*}
  \int_{a}^{b} \int_1^\infty e^{- t\left( 1 - r + \frac{(hr-\alpha)^2}{1-r} - \delta\right)}  \dd h\dd r
  \leq b \int_a^b \int_{r-\alpha}^\infty e^{-t(1-r + \frac{v^2}{1-r} - \delta)} \dd v \dd r\\
  \leq C t^{-1/2} \left(\int_a^\alpha e^{-t(1-r)} \dd r + \int_\alpha^b e^{-t(1-r + \frac{(r-\alpha)^2}{1-r})}\right) \dd r,
\end{multline*}
using \eqref{basic} to bound the integrals over $v$. Hence, one straightforwardly obtains that
\begin{multline}
  \label{eqn:partie1}
  \limsup_{t \to \infty} \frac{1}{t} \log \int_{a}^{b} \int_1^\infty e^{- t\left( 1 - r + \frac{(hr-\alpha)^2}{1-r} - \delta\right)}  \dd h\dd r \\
  \leq \delta -\min(1-\alpha,g_\alpha(b)) < -2 \gamma(1-\alpha),
\end{multline}
for $\delta > 0$ small enough, where $g_\alpha$ is the function defined in \eqref{galpha}, which attains its maximum at $v_\alpha$ with value $2 \gamma(1-\alpha)$.

Similarly, as $[a,b] \times [-\beta,1]$ is compact, we also have
\begin{multline*}
  \limsup_{t \to \infty} \frac{1}{t} \log \int_{a}^{b} \int_{-\gamma}^1 e^{- t\left( 1 - r + \frac{(hr-\alpha)^2}{1-r} + 2\bar{\Psi}_\beta(h) - \delta\right)}  \dd h\dd r \\
  \leq \delta - \inf_{\substack{r \in [a,b]\\ h \in [-\beta,1]}}  1 - r + \frac{(hr-\alpha)^2}{1-r} + 2\bar{\Psi}_\beta(h)
  \leq \delta - \inf_{\substack{r \in [a,b]\\ h \in [-\beta,1]}}  1 - r + \frac{(hr-\alpha)^2}{1-r} + 2\sqrt{2}\gamma (1-h).
\end{multline*}
The function $(h,r) \in (-\infty,1] \times [\epsilon,1] \mapsto 1 - r + \frac{(hr-\alpha)^2}{1-r} + 2\sqrt{2}\gamma (1-h)$ attaining its unique minimum at $(v_\alpha,1)$, we conclude again that, choosing $\delta > 0$ small enough, we have
\begin{equation}
    \label{eqn:partie2}
     \limsup_{t \to \infty} \frac{1}{t} \log \int_{a}^{b} \int_{-\gamma}^1 e^{- t\left( 1 - r + \frac{(hr-\alpha)^2}{1-r} + 2\bar{\Psi}_\beta(h) - \delta\right)}  \dd h\dd r  < -2 \gamma (1 -\alpha).
\end{equation}

Finally, choosing $\beta > 0$ large enough so that the function $h \mapsto  \frac{(hr-\alpha)^2}{1-r} + 2h^2$ is strictly decreasing on $(-\infty,-\beta]$, we have
\begin{multline*}
  \int_{a}^{b} \int_{-\infty}^{-\beta} e^{- t\left( 1 - r + \frac{(hr-\alpha)^2}{1-r} + 2h^2 - \delta\right)}  \dd h\dd r
  \leq \int_a^b e^{-t \left( 1 - r + \frac{(-\beta r -\alpha)^2}{1-r} - \delta \right)} \dd r \int_{-\infty}^{-\beta} e^{- 2 h^2 t} \dd h\\
  \leq C t^{-1/2} e^{-2 \beta^2 t} \int_a^b  e^{-t \left( 1 - r + \frac{(-\beta r -\alpha)^2}{1-r} - \delta \right)} \dd r \leq C t^{-1/2} e^{-t(2 \beta^2 - \delta)},
\end{multline*}
which, if we choose $\beta$ large enough, will be smaller than $e^{-t(2 \gamma (1-\alpha) + \eta)}$ for some $\eta > 0$, for all $t$ large enough. Using this estimate in combination with \eqref{eqn:partie1} and \eqref{eqn:partie2} allows us, by \eqref{eqn:theAim}, to show that
\[
  \limsup_{t \to \infty} \frac{1}{t} \log U_2(\sqrt{2}\alpha t, t,[at,bt]) < -2 \gamma (1-\alpha),
\]
which completes the proof of \eqref{upbdmediatime} and \eqref{upbdlargetime+}.
\end{proof}

The proof of Lemma \ref{small2ndterm} is then a combination of Lemmas \ref{lem:upbdsmalltime} and \ref{lem:upbdmediatetime}.

\subsubsection{Tightness of the normalized first splitting time}

We now precise the estimates on $\P\left(|\tau - v_\alpha t| > A \sqrt{t} , M_t \leq \sqrt{2}\alpha t\right)$, bounding this quantity as $t \to \infty$ then $A \to \infty$.

\begin{lemma}
\label{small2ndtermbadtime}
Given $\alpha \in (-\gamma,1)$, we have
\begin{align}
  \lim_{A\to \infty}& \limsup_{t \to \infty} e^{2 \gamma (1-\alpha)t} t^{-3\gamma/2} U_2(\sqrt{2}\alpha t, t, [0, v_\alpha t-A\sqrt{t}])=0;\label{upbdsmallbadtime-}\\
  \lim_{A\to \infty}& \limsup_{t \to \infty} e^{2 \gamma (1-\alpha)t} t^{-3\gamma/2} U_2(\sqrt{2}\alpha t, t, [v_\alpha t+A\sqrt{t},t])=0.\label{upbdsmallbadtime+}
\end{align}
\end{lemma}

As a first step, we show that with high probability, $|\tau - v_\alpha t| = o(t^{1/2}\log t)$ conditioned on the maximal displacement being small.
\begin{lemma}\label{lem:small2ndtermbadtime1}
Let $\alpha \in (-\gamma,1)$. There exists $\epsilon_\alpha > 0$ such that for all $\epsilon\in(0, \epsilon_\alpha)$, for all $\rho > 0$ we have
\begin{align}
  \limsup_{t \to \infty} t^\rho e^{2 \gamma(1-\alpha)t} U_2(\sqrt{2}\alpha t, t, [(v_\alpha-\epsilon )t, v_\alpha t-\sqrt{t}\log t])&=0;\label{upbdlargebadtime-}\\
  \limsup_{t \to \infty} t^\rho e^{2 \gamma(1-\alpha)t} U_2(\sqrt{2}\alpha t, t, [v_\alpha t+\sqrt{t}\log t, (v_\alpha +\epsilon )t])&=0.\label{upbdlargebadtime+}
\end{align}
\end{lemma}

\begin{proof}
The two formulas being proved in a very similar way, we only prove the first one. Note that without loss of generality, one can choose $\epsilon>0$ small enough that $v_\alpha - 2\epsilon > \min(\alpha,0)$. By definition of $U_2$, we have
\begin{align*}
  &U_2(\sqrt{2}\alpha t, t, [(v_{\alpha} - \epsilon) t, v_\alpha t - \sqrt{t}\log t ])\\
  = &\int_{(v_\alpha-\epsilon)t}^{v_\alpha t - \sqrt{t}\log t} \dd s \int_\R \frac{\dd z}{\sqrt{2 \pi t}} e^{-(t-s)- \frac{(\sqrt{2} s +  z-\sqrt{2}\alpha t)^2}{2(t-s)}} u(\sqrt{2}s + z,s)^2\\
  \leq & t^{1/2}\int_{v_\alpha - \epsilon}^{v_\alpha - \frac{\log t}{\sqrt{t}}} \dd u \int_\R\dd z e^{- t(1-u) + \frac{ \left(z + \sqrt{2}t( u -\alpha) \right)^2}{2t(1-u)}} u(\sqrt{2}u t + z,ut)^2\\
  \leq & t^{1/2} \int_{v_\alpha-\epsilon}^{v_\alpha - \frac{\log t}{\sqrt{t}}} \dd u e^{-tg_\alpha (u)} \int_\R \dd z e^{-\frac{z(2\sqrt{2}t (u-\alpha) + z)}{2t(1-u)}} u(\sqrt{2}u t + z,ut)^2,
\end{align*}
with $g_\alpha$ the function defined in \eqref{galpha}. Using \eqref{Taylorg}, there exists $c>0$ such that for all $r \in [v_\alpha - \epsilon,v_\alpha]$, $g_\alpha(r) \leq g_\alpha(v_\alpha) - c (r-v_\alpha)^2.$
Thus
\begin{multline*}
  e^{2 \gamma(1-\alpha) t} U_2(\sqrt{2}\alpha t, t, [(v_{\alpha} - \epsilon) t, v_\alpha t - \sqrt{t}\log t ])\\
  \leq t \int_{-\epsilon}^{-\frac{\log t}{t^{1/2}}} \dd v e^{- c t v^2} \int_\R\dd z e^{-\frac{z(2\sqrt{2}t (v_\alpha + v-\alpha) + z)}{2t(1-v_\alpha - v)}} u(\sqrt{2}(v_\alpha + v) t+z,(v_\alpha + v)t)^2.
\end{multline*}

We now use Lemma \ref{CHbdu}, i.e. that for all $\delta > 0$ there exists $c_\delta > 0$ such that for all $t \geq 1$ and $z \in \R$, we have $u(m_t - z,t) \leq c_\delta e^{-\sqrt{2}\gamma (1-\delta)z_+}$. Therefore, up to a change of variables, for all $v \in [-\epsilon,0]$, writing $a_t(v) = \frac{3}{2\sqrt{2}} \log ((v_\alpha + v)t)$, we have
\begin{align*}
  &\int_\R\dd z  e^{-\frac{z(2\sqrt{2}t (v_\alpha + v-\alpha) + z)}{2t(1-v_\alpha - v)}} u(\sqrt{2}(v_\alpha + v) t+z,(v_\alpha + v)t)^2\\
  \leq & \int_\R \dd y  e^{\frac{(y+a_t(v))(2\sqrt{2}t (v_\alpha + v-\alpha) - (y+a_t(v)))}{2t(1-v_\alpha - v)}} u(m_{(v_\alpha + v)t} - y,(v_\alpha + v)t)^2\\
  \leq & c_\delta \int_\R \dd y  e^{\frac{(y+a_t(v))(2\sqrt{2}t (v_\alpha + v-\alpha) -(y+a_t(v)))}{2t(1-v_\alpha - v)}} e^{-2 \sqrt{2} \gamma (1-\delta) y_+}.
\end{align*}
As a result, we get
\begin{multline}
  \label{eqn:step}
    e^{2 \gamma(1-\alpha) t} U_2(\sqrt{2}\alpha t, t, [(v_{\alpha} - \epsilon) t, v_\alpha t - \sqrt{t}\log t ])\\
    \leq C t^{1/2} \int_{-\epsilon}^{-\frac{\log t}{t^{1/2}}} \dd v e^{- c t v^2} \int_\R\dd y e^{\frac{(y+a_t(v))(2\sqrt{2}t (v_\alpha + v-\alpha) -(y+a_t(v)))}{2t(1-v_\alpha - v)}} e^{-2 \sqrt{2} \gamma (1-\delta) y_+}.
\end{multline}
We now bound this quantity in two different ways for $y \geq 0$ and $y \leq 0$.

We first observe that  for all $v \in [-\epsilon,0]$, using that $v_\alpha >\alpha + 2 \epsilon$,
\begin{multline}
  \int_{-\infty}^0 \dd y  e^{\frac{(y+a_t(v))(2\sqrt{2}t (v_\alpha + v-\alpha) -(y+a_t(v)))}{2t(1-v_\alpha - v)}} \leq \int_{-\infty}^0 \dd y  e^{\frac{(y+a_t(v))(2\sqrt{2}t (v_\alpha + v-\alpha))}{2t(1-v_\alpha - v)}}\\
  \leq  \frac{1 - v_\alpha - v}{\sqrt{2}(v_\alpha + v -\alpha)} \left( (v+v_\alpha) t \right)^{\frac{3}{2} \frac{v_\alpha + v -\alpha}{1-v_\alpha - v}}
  \leq  \frac{1-v_\alpha + \epsilon}{2 \sqrt{2}\epsilon} \left( v_\alpha  t\right)^{\frac{3}{2}\frac{v_\alpha -\alpha}{(1 - v_\alpha)}}. \label{eqn:apoint}
\end{multline}
Similarly, we have
\begin{multline}
  \int_{0}^\infty \dd y  e^{\frac{(y+a_t(v))(2\sqrt{2}t (v_\alpha + v-\alpha) -(y+a_t(v)))}{2t(1-v_\alpha - v)}} e^{-2 \sqrt{2} \gamma (1-\delta) y}\\
  \leq  e^{2 \sqrt{2}\gamma (1-\delta) a_t(v)} \int_{a_t(v)}^\infty \dd x e^{x \left( \sqrt{2}\frac{(v_\alpha -\alpha)}{(1-v_\alpha)} - 2\sqrt{2}\gamma (1-\delta) \right)}
  \leq  \frac{1}{\sqrt{2}\gamma (1 - 2\delta)} (v_\alpha t)^{3 \gamma (1-2\delta)/2}, \label{eqn:quisaitattendre}
\end{multline}
for all $\delta > 0$ small enough, using that $v_\alpha -\alpha = \frac{\gamma}{\sqrt{2}} (1 -\alpha) = \gamma (1-v_\alpha)$.

Hence, plugging \eqref{eqn:apoint} and \eqref{eqn:quisaitattendre} into \eqref{eqn:step}, we deduce that there exist $C>0$ and $\rho > 0$ so that for all $t \geq 1$ large enough,
\begin{multline*}
\qquad  e^{2 \gamma(1-\alpha) t} U_2(\sqrt{2}\alpha t, t, [(v_{\alpha} - \epsilon) t, v_\alpha t - \sqrt{t}\log t ])\\
  \leq C t^\rho \int_{-\epsilon}^{- \frac{\log t}{t^{1/2}}} \dd v e^{-ct v^2} \leq C t^{\rho}e^{-c(\log t)^2},\qquad
\end{multline*}
which concludes the proof of \eqref{upbdlargebadtime-}. 
\end{proof}

\begin{proof}[Proof of Lemma \ref{small2ndtermbadtime}]
By Lemmas \ref{small2ndterm} and \ref{lem:small2ndtermbadtime1}, to prove Lemma \ref{small2ndtermbadtime}, it is enough to bound for all $t$ large enough, the quantities $U_2(\sqrt{2}\alpha t, t, [v_{\az}t-\sqrt{t}\log t, v_{\az}t-A\sqrt{t}])$ and $U_2(\sqrt{2}\alpha t, t, [v_{\az}t+A\sqrt{t}v_{\az}t-\sqrt{t}\log t])$ by $M(A)e^{-2\gamma(1-\alpha)t}t^{-3\gamma/2}$, with $A \mapsto M(A)$ a positive function converging to $0$ as $A \to \infty$.

The proofs of \eqref{upbdsmallbadtime-} and \eqref{upbdsmallbadtime+} being very similar and symmetric, we only prove the second one. We write
\begin{align*}
  &U_2(\sqrt{2}\alpha t, t, [v_{\az}t+A\sqrt{t}, v_{\az}t+\sqrt{t}\log t])\\
  \leq &t^{-1/2}\int_{v_\alpha t + A \sqrt{t}}^{v_\alpha t + \sqrt{t} \log t} \dd s \int_\R \dd z e^{- (t-s) + \frac{(z-m_s)^2}{2(t-s)}} u(m_s + z, s)^2\\
  \leq &t^{1/2} e^{-2 \gamma(1-\alpha) t}
    \int_{At^{-1/2}}^{t^{-1/2}\log t } \!\!\!\!\!\!\!\! \dd v e^{-ctv^2}
    \int_\R\dd y e^{\frac{(y+a_t(v))(2\sqrt{2}t (v_\alpha + v-\alpha) -(y+a_t(v)))}{2t(1-v_\alpha - v)}} e^{-2 \sqrt{2} \gamma (1-\delta) y_+},
 %   \int_\R\dd z e^{-\frac{z(\sqrt{2}t (v_\alpha + v-\alpha) + z)}{2t(1-v_\alpha - v)}} u(\sqrt{2}(v_\alpha + v) t+z,(v_\alpha + v)t)^2,
\end{align*}
with the same computations as the ones used to obtain \eqref{eqn:step}, using Lemma~\ref{CHbdu}.

Using that $|v| \leq t^{-1/2}\log t$, hence that $a_t(v) = \frac{3}{2\sqrt{2}} \log ((v_\alpha + v)t) = a_t(0) + o_t(1)$, we obtain, for all $t$ large enough:
\begin{align}
  &U_2(\sqrt{2}\alpha t, t, [v_{\az}t+A\sqrt{t}, v_{\az}t+\sqrt{t}\log t])\nonumber\\
  \leq & 2 c_\delta t^{1/2}e^{-2 \gamma(1-\alpha) t} \int_{At^{-1/2}}^{t^{-1/2}\log t } \!\!\!\!\!\!\!\! \dd v e^{-ctv^2} \int_\R\dd y e^{\frac{(y + a_t)(2 \sqrt{2} t (v_\alpha + v -\alpha) - (y + a_t)}{2t(1-v_\alpha - v)}} e^{-2 \sqrt{2}\gamma(1-\delta) y_+},\label{notset}
\end{align}
where $a_t = a_t(0)= \frac{3}{2} \log (v_\alpha t)$.

We then compute for all $|v| < t^{-1/2} \log t$,
\begin{multline*}
  \int_{-\infty}^0 \dd y e^{\frac{(y + a_t)(2 \sqrt{2} t (v_\alpha + v -\alpha) -(y + a_t))}{2t(1-v_\alpha - v)}}
  \leq \int_{-\infty}^0 \dd y e^{\sqrt{2}\frac{(y+a_t)(v_\alpha + v -\alpha)}{(1-v_\alpha - v)}}\\
  \leq \exp\left(  \sqrt{2} a_t  \frac{v_\alpha + v-\alpha}{1 - v_\alpha - v} \right) \leq \exp \left( \sqrt{2} a_t \left(\frac{v_\alpha -\alpha}{1-v_\alpha} + C v\right) \right),
\end{multline*}
for all $t$ large enough, using Taylor's expansion. Hence, with $(v_\alpha-\alpha)/(1-v_\alpha) = \gamma$, there exists $C>0$ such that for all $t$ large enough,
\begin{equation}
  \label{eqn:fstPart}
  \int_{-\infty}^0 \dd y e^{\frac{(y + a_t)(2 t (v_\alpha + v -\alpha) - (y+a_t))}{t(1-v_\alpha - v)}} \leq C (v_\alpha t)^{3\gamma/2}.
\end{equation}
Similarly, we have
\begin{align*}
  &\int_0^{\infty} \dd y e^{\frac{(y + a_t)(2 \sqrt{2} t (v_\alpha + v -\alpha) - (y + a_t))}{2t(1-v_\alpha - v)}} e^{-2 \sqrt{2}\gamma(1-\delta) y}\\
  \leq &e^{2 \sqrt{2}\gamma(1-\delta)a_t} \int_{a_t}^\infty \dd x e^{\frac{x(2\sqrt{2}t (v_\alpha + v -\alpha) - x)}{2t(1-v_\alpha - v)} - 2 \sqrt{2}\gamma(1-\delta) x}\\
  \leq &(v_\alpha t)^{3\gamma (1 - \delta)} \int_{a_t}^\infty \dd x e^{\sqrt{2}x \left( \frac{v_\alpha + v -\alpha}{1-v_\alpha - v} -2 \gamma(1-\delta)\right)}
  \leq C (v_\alpha t)^{3 \gamma(1-\delta)} (v_\alpha t)^{\frac{3}{2}\frac{v_\alpha - v -\alpha}{1-v_\alpha}  - 3 \gamma (1-\delta)}.
\end{align*}
Hence, using that $\frac{v_\alpha - v -\alpha}{1-v_\alpha} = \gamma + O(t^{-1/2}\log t)$, we obtain that for all $t$ large enough
\begin{equation}
  \label{eqn:sndPart}
  \int_{0}^\infty \dd y e^{\frac{(y + a_t)(2 t (v_\alpha + v -\alpha) - (y + a_t))}{t(1-v_\alpha - v)}} \leq C (v_\alpha t)^{3\gamma/2}.
\end{equation}
As a result, with \eqref{eqn:fstPart} and \eqref{eqn:sndPart}, \eqref{notset} becomes
\begin{equation*}
  U_2(\sqrt{2}\alpha t, t, [v_{\az}t+A\sqrt{t}, v_{\az}t+\sqrt{t}\log t])
  \leq C t^{3\gamma/2} e^{-2 \gamma(1-\alpha)t} \int_{A }^{\infty} e^{-cw^2} \dd w.
\end{equation*}
By dominated convergence, the proof of \eqref{upbdsmallbadtime+} is now complete.
\end{proof}

\subsubsection{Tightness of the centred splitting position}

To complete the proof of Lemma \ref{badtime+badposition}, we prove that the position at which the first splitting occurs $X_\emptyset(\tau)$ is tightly concentrated around the position $\sqrt{2}\alpha t - m_{t-\tau}$, on the event $|\tau - v_\alpha t| \leq A \sqrt{t}$.

\begin{lemma}\label{goodtimebadposition+}
For $1-\sqrt{2}<\alpha<1$ and for any fixed $A>0$,
\begin{align}
\lim_{K\to\infty}\lim_{t\to\infty}\frac{e^{2\gamma(1-\alpha)t}}{t^{3\gamma/2}}U_2(\sqrt{2}\alpha t, t,[v_\alpha t-A\sqrt{t}, v_\alpha+A\sqrt{t}],[-K,K]^c)&=0.\label{upbdsmallpositionA}
\end{align}
\end{lemma}

\begin{proof}
Let $K > 0$ and $A > 0$. We observe that with similar computations as in the proof of Lemma \ref{small2ndtermbadtime}, setting $a_t = \frac{3}{2\sqrt{2}} \log (v_\alpha t)$, we have
\begin{align*}
  &e^{2 \gamma(1-\alpha) t} U_2(\sqrt{2}\alpha t, t,[v_\alpha t-A\sqrt{t}, v_\alpha t+A\sqrt{t}],[-K,K]^c)\\
\leq & C t^{1/2}  \int_{-At^{-1/2}}^{At^{-1/2}} \dd v e^{- c t v^2} \int_{[-K,K]^c} \dd y e^{\frac{(y + a_t)(2 \sqrt{2} t (v_\alpha + v -\alpha) - (y + a_t))}{2t(1-v_\alpha - v)}} u(m_{(v_\alpha + v)t}-y,(v_\alpha + v)t) \\
  \leq &C t^{1/2}  \int_{-At^{-1/2}}^{At^{-1/2}} \dd v e^{- c t v^2} \int_{[-K,K]^c} \dd y e^{\frac{(y + a_t)(2 \sqrt{2} t (v_\alpha + v -\alpha) - (y + a_t))}{2t(1-v_\alpha - v)}}  e^{-2 \sqrt{2}\gamma(1-\delta) y_+},
\end{align*}
where we used again Lemma \ref{CHbdu}.

We then observe, with similar computations as in the proof of Lemma \ref{small2ndtermbadtime} again that
\begin{align*}
  &\int_{-\infty}^{-K} \dd y e^{\frac{(y + a_t)(2 \sqrt{2} t (v_\alpha + v -\alpha) - (y + a_t))}{2t(1-v_\alpha - v)}} \leq t^{3\gamma/2} e^{-(\gamma-\delta) K},\\
  &\int_{K}^\infty \dd y e^{\frac{(y + a_t)(2 \sqrt{2} t (v_\alpha + v -\alpha) - (y + a_t))}{2t(1-v_\alpha - v)}} e^{-2 \sqrt{2}\gamma(1-\delta) y} \leq t^{3\gamma/2} e^{-(\gamma-\delta) K},
\end{align*}
using that for all $t$ large enough, $\left| \frac{v_\alpha - v-\alpha}{1 - v_\alpha} - \gamma\right| \leq \delta$. Therefore, letting $t \to \infty$ then $K \to \infty$, we obtain, for all $A > 0$, that \eqref{upbdsmallpositionA} holds.
\end{proof}

Lemma \ref{badtime+badposition} is then a consequence of Lemmas \ref{small2ndterm}, \ref{small2ndtermbadtime} and \ref{goodtimebadposition+}.

\subsection{Proof of Lemma \ref{largetime}}
\label{lems3}

Similarly to the previous section, it is enough to prove Lemma \ref{largetime} for $\phi \equiv 0$ by a straightforward domination argument. The proof is obtained in a similar, but slightly simple fashion.

\begin{proof}
Let $\alpha<-\gamma$ here. Note that by change of variable $y=\sqrt{2} as$ and \eqref{DSupbdu}, for any $\epsilon>0$ and $A\geq t_{\epsilon,\beta}$, with $\beta = K\alpha$,
\begin{align*}
U_2(\sqrt{2}\alpha t, t, [A, t])=&\int_A^t\dd s\int_\R  \frac{e^{-(t-s)-\frac{(\sqrt{2}\alpha t-\sqrt{2}as)^2}{2(t-s)}}}{\sqrt{2\pi (t-s)}}u^2(\sqrt{2}as,s)\sqrt{2}s\dd a\\
\leq &\Sigma_1(A, t)+\Sigma_2(A,t)+\Sigma_3(A,t)+\Sigma_4(A,t),
\end{align*}
where
\begin{align*}
\Sigma_1(A, t):=&\int_A^t\dd s\int_{1}^\infty  \frac{e^{-(t-s)-\frac{(\sqrt{2}\alpha t-\sqrt{2}as)^2}{2(t-s)}}}{\sqrt{2\pi (t-s)}}\sqrt{2}s\dd a,\\
\Sigma_2(A, t):=&\int_A^t\dd s\int_{-\gamma}^1  \frac{e^{-(t-s)-\frac{(\sqrt{2}\alpha t-\sqrt{2}as)^2}{2(t-s)}}}{\sqrt{2\pi (t-s)}}e^{-4\gamma(1-a)s+2\epsilon s}\sqrt{2}s\dd a,\\
\Sigma_3(A, t):=&\int_A^t\dd s\int_{K\az}^{-\gamma}  \frac{e^{-(t-s)-\frac{(\sqrt{2}\alpha t-\sqrt{2}as)^2}{2(t-s)}}}{\sqrt{2\pi (t-s)}}e^{-2(1+a^2)s+2\epsilon s}\sqrt{2}s\dd a,\\
\Sigma_4(A, t):=&\int_A^t\dd s\int_{-\infty}^{K\az}  \frac{e^{-(t-s)-\frac{(\sqrt{2}\alpha t-\sqrt{2}as)^2}{2(t-s)}}}{\sqrt{2\pi (t-s)}}e^{-2a^2s}\sqrt{2}s\dd a.
\end{align*}
Recall $g_{\az}$ from \eqref{galpha}. By change of variables $z=\sqrt{2}as-\sqrt{2}\az t$ and $s=ut$ and by \eqref{basic}, one gets that
\begin{align*}
 \Sigma_1(A, t)=&\int_{A/t}^1 te^{-t(1-u)}\dd u \int_{\sqrt{2}ut-\sqrt{2}\az t}^\infty \frac{e^{-\frac{z^2}{2t(1-u)}}}{\sqrt{2\pi(1-u)t}}\dd z\\
 \leq &\int_{A/t}^1 te^{-t(1-u)}\frac{\sqrt{t(1-u)}}{\sqrt{2}(u-\az)t}e^{-\frac{(u-\az)^2}{1-u}t}\dd u\leq \frac{\sqrt{t}}{|\az|}\int_{A/t}^1e^{-tg_{\az}(u)}\dd u.
\end{align*}
Clearly, $g_{\az}(h)=g_{\az}(0)+g'_{\az}(0)h+o(h)$ as $|h|\to0$. Note that $g_{\az}(0)=1+\az^2$ and $g'_{\az}(0)=(\az-1)^2-2>0$ for $\az<1-\sqrt{2}$. Note that, for any $u\in[\frac{A}{\sqrt{t}},1]$,
\[
g_{\az}(u)\geq g_{\az}\left(\frac{A}{\sqrt{t}}\right)=g_{\az}(0)+(g'_{\az}(0)+o_t(1))\frac{A}{\sqrt{t}},
\]
which implies that, for $t$ sufficiently large,
\[
\frac{\sqrt{t}}{|\az|}\int_{\frac{A}{\sqrt{t}}}^1e^{-tg_{\az}(u)}\dd u\leq \frac{e^{-(1+\az^2)t}}{\sqrt{t}|\az|} te^{-\frac{Ag_{\az}'(0)}{2}\sqrt{t}}=o_t(1) \frac{e^{-(1+\az^2)t}}{\sqrt{t}}.
\]
On the other hand, since $g_{\az}(h)=g_{\az}(0)+g'_{\az}(0)h+o(h)$ as $|h|\to0$, then
\[
\frac{\sqrt{t}}{|\az|}\int_{\frac{A}{t}}^{\frac{A}{\sqrt{t}}}e^{-tg_{\az}(u)}\dd u=\frac{e^{-(1+\az^2)t}}{\sqrt{t}|\az|} \int_{\frac{A}{t}}^{\frac{A}{\sqrt{t}}} te^{-t(g'_{\az}(0)+o_t(1))u}\dd u=o_{t,A}(1)\frac{e^{-(1+\az^2)t}}{\sqrt{t}|\az|}.
\]
Thus $\Sigma_1(A, t)\leq o_{t,A}(1)\frac{e^{-(1+\az^2)t}}{\sqrt{t}|\az|}. $ Next, we shall handle  $\Sigma_2(A, t)$. If $\alpha<-2\gamma$, then $\gamma s-(\alpha+2\gamma)t>0$. So, by change of variable $z=as-\alpha t+2\gamma(s-t)$ and \eqref{basic},
\begin{align*}
\Sigma_2(A, t)=&\int_A^t e^{-(t-s)-4\gamma s+4\gamma\alpha t+4\gamma^2(t-s)+2\epsilon s}\dd s\int_{\gamma s-(\alpha+2\gamma)t}^{(1+2\gamma)s-(\alpha+2\gamma)t}\frac{e^{-\frac{z^2}{t-s}}}{\sqrt{2\pi (t-s)}}\sqrt{2}\dd z\cr \leq &\int_A^t e^{-(t-s)-4\gamma s+4\gamma\alpha t+4\gamma^2(t-s)+2\epsilon s} \frac{\sqrt{t-s}}{\sqrt{2}\left(\gamma s-(\alpha+2\gamma)t\right)}e^{-\frac{(\gamma s-(\alpha+2\gamma)t)^2}{t-s}}\dd s\\
=&\int_{A/t}^1\frac{\sqrt{t(1-u)}}{\sqrt{2}\left(\gamma u-(\alpha+2\gamma)\right)}e^{-t[\frac{(\alpha+\gamma)^2}{1-u}-2\gamma(1+\gamma)(1-u)+2-2\alpha\gamma-2\epsilon u]}\dd u\\
\leq & \frac{\sqrt{t}}{|\alpha|-2\gamma}\int_{A/t}^1e^{-tg_{\alpha, \epsilon}(u)}\dd u,
\end{align*}
where $g_{\alpha,\epsilon}(u)=\frac{(\alpha+\gamma)^2}{1-u}-(1+\gamma^2)(1-u)+2-2\alpha\gamma-2\epsilon u$. Observe that for $\epsilon\in(0,1/2)$ and $u\in(0,1)$,
\[
g'_{\alpha,\epsilon}(u)=\frac{(\alpha+\gamma)^2}{(1-u)^2}+(1+\gamma^2)-2\epsilon\geq L_\epsilon:=(\alpha+\gamma)^2+(1+\gamma^2)-2\epsilon,
\]
and that $g_{\alpha,\epsilon}(0)=\alpha^2+1$. Then, for any $h\in(0,1)$,
\[
\min_{u\in[h,1]}g_{\alpha,\epsilon}(u)\geq g_{\alpha,\epsilon}(h)\geq\alpha^2+1+L_\epsilon h.
\]
This implies that if $\alpha<-2\gamma$, then
\[
\Sigma_2(A, t)\leq  \frac{\sqrt{t}}{|\alpha|-2\gamma}\int_{A/t}^1e^{-t(\alpha^2+1+L_\epsilon u)}\dd u=\frac{e^{-(1+\alpha^2)t}}{\sqrt{t}(|\alpha|-2\gamma)} \int_{A/t}^1 e^{-L_\epsilon ut}t\dd u,
\]
which is $o_A(1)\frac{e^{-(1+\alpha^2)t}}{\sqrt{t}}$. If $-2\gamma\leq\alpha<-\gamma$, then
\begin{align*}
\Sigma_2(A, t)\leq &\int_A^t \left(\int_{-\gamma}^1  \frac{e^{-(t-s)-\frac{(\sqrt{2}\alpha t-\sqrt{2}as)^2}{2(t-s)}}}{\sqrt{2\pi (t-s)}}e^{-2\gamma(1-a)s+2\epsilon s}\sqrt{2}sda\right)\dd s\\
=&\int_A^t \left(\int_{-(\alpha+\gamma)t}^{(1+\gamma)s-(\alpha+\gamma)t}\frac{e^{-\frac{z^2}{t-s}}}{\sqrt{2\pi (t-s)}}\sqrt{2}\dd z\right)e^{\gamma^2(t-s)+2\gamma\alpha t-(t-s)-2\gamma s+2\epsilon s}\dd s\\
\leq &\int_A^t \frac{\sqrt{t-s}}{-(\alpha+\gamma)t}e^{-\frac{(\alpha+\gamma)^2t^2}{t-s}+\gamma^2(t-s)+2\gamma\alpha t-(t-s)-2\gamma s+2\epsilon s}\dd s\cr= & \int_{A/t}^1\frac{\sqrt{t(1-u)}}{|\alpha|-\gamma} e^{-t h(u)}\dd u,
\end{align*}
where in the first equality, we change variable $z=sa-\alpha t+\gamma (s-t)$, the second inequality holds by \eqref{basic} and $h(u)=\frac{(\alpha+\gamma)^2}{1-u}-2\epsilon u+(1+\alpha^2)-(\alpha+\gamma)^2$. Note that for any $\epsilon\in\left(0,\frac{(\alpha+\gamma)^2}{2}\right)$ and $u\in(0,1)$,
\[
h'(u)=\frac{(\alpha+\gamma)^2}{(1-u)^2}-2\delta\geq \tilde{L}_\epsilon:=(\alpha+\gamma)^2-2\epsilon>0,
\]
with $h(0)=\alpha^2+1$. It hence follows that if $-2\gamma\leq\alpha<-\gamma$, then
\begin{align*}
\Sigma_2(A, t)\leq \int_{A/t}^1\frac{\sqrt{t(1-u)}}{|\alpha|-\gamma} e^{-t(\alpha^2+1+\tilde{L}_\epsilon u)}\dd u
=o_A(1)\frac{e^{-(1+\alpha^2)t}}{\sqrt{t}}.
\end{align*}
For $\Sigma_3(A, t)$, one sees that
\begin{align*}
  \Sigma_3(A, t)
  =&\int_A^t e^{-t-s+2\epsilon s-\frac{2\alpha^2t^2}{2t-s}}\dd s\int_{K\alpha}^{-\gamma} e^{-\frac{s(2t-s)}{t-s}(a-\frac{\alpha t}{2t-s})^2}\frac{s}{\sqrt{\pi(t-s)}}da\\
  \leq & \int_A^t\frac{\sqrt{s}}{\sqrt{2t-s}} e^{-t-s+2\epsilon s-\frac{2\alpha^2t^2}{2t-s}}\dd s\leq \frac{e^{-(1+\alpha^2)t}}{\sqrt{t}}\int_A^t \sqrt{s}e^{-(1-2\epsilon)s}\dd s,
\end{align*}
which is $o_A(1)\frac{e^{-(1+\alpha^2)t}}{\sqrt{t}}$ as long as $\epsilon\in(0,1/2)$. On the other hand,
\begin{align*}
\Sigma_4(A, t)=&\int_A^t\dd s\int_{-\infty}^{K\alpha}  \frac{e^{-(t-s)-\frac{(\sqrt{2}\alpha t-\sqrt{2}as)^2}{2(t-s)}}}{\sqrt{2\pi (t-s)}}e^{-2a^2s}\sqrt{2}sda\\
=&\int_A^t e^{-t+s-\frac{2\alpha^2t^2}{2t-s}}\dd s\int_{-\infty}^{K\alpha} e^{-\frac{s(2t-s)}{t-s}(a-\frac{\alpha t}{2t-s})^2}\frac{s}{\sqrt{\pi(t-s)}}da\\
=&\int_A^t e^{-t+s-\frac{2\alpha^2t^2}{2t-s}}\sqrt{\frac{s}{2t-s}}\dd s\int_{-\infty}^{K\alpha-\frac{\alpha t}{2t-s}} e^{-\frac{s(2t-s)}{t-s}z^2} \frac{\dd z}{\sqrt{\pi\frac{t-s}{s(2t-s)}}}.
\end{align*}
Choose $K>1$ such that $(K-1)|\alpha|>1$ and $K\alpha-\frac{\alpha t}{2t-s}<-1$. Then by \eqref{basic},
\begin{align*}
\Sigma_4(A, t)\leq &\int_A^t e^{-t+s-\frac{2\alpha^2t^2}{2t-s}}\sqrt{\frac{s}{2t-s}}\frac{\sqrt{t-s}}{\sqrt{s(2t-s)}}e^{-\frac{s(2t-s)}{t-s}}\dd s\cr
=&\int_A^t \frac{\sqrt{t-s}}{2t-s} e^{-t-\frac{2\az^2t^2}{2t-s}} e^{-s-\frac{s^2}{t-s}}\dd s \leq \frac{e^{-(1+\alpha^2)t}}{\sqrt{t}}\int_A^t e^{-s}\dd s,
\end{align*}
as $\frac{1}{2t}\leq \frac{1}{2t-s}\leq \frac{1}{t}$ and $\sqrt{t-s}\leq \sqrt{t}$. Therefore, $\Sigma_4(A, t)=o_A(1)\frac{e^{-(1+\alpha^2)t}}{\sqrt{t}}$.
\end{proof}

\subsection{Proof of Lemma \ref{badtimecA}}
\label{lems4}

Using again a domination argument, it is enough to prove Lemma \ref{badtimecA} for $\phi \equiv 0$. We decompose it into the two following lemmas, that we prove one by one.

\begin{lemma}\label{badtimec}
\begin{align}\label{largebadtimec}
\lim_{A\to\infty}\lim_{t\to\infty}\frac{e^{(1+\gamma^2)t}}{t^{3\gamma/4}}U_2(-\sqrt{2}\gamma t, t, [A\sqrt{t},t])=&0;\\\label{smallbadtimec}
\lim_{A\to\infty}\lim_{t\to\infty}\frac{e^{(1+\gamma^2)t}}{t^{3\gamma/4}}U_2(-\sqrt{2}\gamma t, t, [0,{\sqrt{t}}/{A}])=&0.
\end{align}
\end{lemma}

\begin{lemma}\label{badpositionc}
For any $A>0$ fixed,
\begin{align}
  \label{badsmallpositionc}
  &\lim_{K\to\infty}\lim_{t\to\infty}\frac{e^{(1+\gamma^2)t}}{t^{3\gamma/4}}U_2(-\sqrt{2}\gamma t, t, [\frac{1}{A}\sqrt{t}, A\sqrt{t}], (-\infty,-K])=0\\
  \label{badlargepositionc}
  \text{and }\quad & \lim_{K\to\infty}\lim_{t\to\infty}\frac{e^{(1+\gamma^2)t}}{t^{3\gamma/4}}U_2(-\sqrt{2}\gamma t, t, [\sqrt{t}/{A}, A\sqrt{t}], [K,\infty))=0.
\end{align}
\end{lemma}

\begin{proof}[Proof of Lemma \ref{badtimec}] \textbf{Proof of \eqref{largebadtimec}:} Observe that
\begin{align}\label{U2F}
&U_2(-\sqrt{2}\gamma t, t, [A\sqrt{t},t])\cr&\quad=
U_2(-\sqrt{2}\gamma t, t, [A\sqrt{t},t], [-K,\infty)])+
U_2(-\sqrt{2}\gamma t, t, [A\sqrt{t},t], (-\infty,-K))
\cr&\quad=:U_{\eqref{U2F}a}+U_{\eqref{U2F}b}.
\end{align}
As $u(m_s+z,s)\leq 1$, one sees that
\begin{align*}
U_{\eqref{U2F}a}\leq &\int_{A\sqrt{t}}^t\dd s\int_{-K}^\infty\dd z \frac{e^{-(t-s)-\frac{(z+m_s+\sqrt{2}\gamma t)^2}{2(t-s)}}}{\sqrt{2\pi (t-s)}}\\
=&\int_{A\sqrt{t}}^t e^{-(t-s)}\dd s\int_{-K+m_s+\sqrt{2}\gamma t}^\infty \frac{e^{-\frac{z^2}{2(t-s)}}}{\sqrt{2\pi (t-s)}},
\end{align*}
which by \eqref{basic} is bounded by
\begin{align*}
&\int_{A\sqrt{t}}^t e^{-(t-s)}\frac{\sqrt{t-s}}{-K+m_s+\sqrt{2}\gamma t}e^{-\frac{(K+m_s+\sqrt{2}\gamma t)^2}{2(t-s)}}\dd s\\
\leq &c_4\frac{e^{-(1+\gamma^2)t}}{\sqrt{t}}\int_{A\sqrt{t}}^t e^{-\frac{2s^2}{t-s}+(\sqrt{2}K+\frac{3}{2}\log s)\frac{s+\gamma t}{t-s}}\dd s.
\end{align*}
For $t$ large enough, one has
\[
\left(\sqrt{2}K+\frac{3}{2}\log s\right)\frac{s+\gamma t}{t-s}\leq
\begin{cases} \frac{s^2}{t-s}, &\text{if }s\in [\sqrt{t}\log t, t];\cr
\frac{3\gamma}{2}\log s+2\sqrt{2} K+\frac{3\sqrt{2}(\log t)^2}{\sqrt{t}}, &\text{if }s\in[A\sqrt{t}, \sqrt{t}\log t],
\end{cases}
\]
which implies that
\begin{align*}
U_{\eqref{U2F}a}\leq & c_5\tfrac{e^{-(1+\gamma^2)t}}{\sqrt{t}} \left(\int_{\sqrt{t}\log t}^te^{-\frac{s^2}{t-s}}\dd s
+e^{2\sqrt{2}K}\int_{A\sqrt{t}}^{\sqrt{t}\log t} s^{3\gamma/2}e^{-\frac{2s^2}{t}}\dd s\right)
=o_t\left(\tfrac{e^{-(1+\gamma^2)t}}{\sqrt{t}}\right).
\end{align*}
On the other hand, for $s$ sufficiently large and $z<-K$, by similar reasonings as in Lemma \ref{small2ndtermbadtime}, we have
for $\delta\in(0,1/2],\, \eta=1-2\delta,\,\epsilon< \frac{\gamma\eta}{1+2\gamma(1-\delta)}$,
\begin{align*}
U_{\eqref{U2F}b}&\leq
c_\delta^2\int_{A\sqrt{t}}^{\epsilon t} \frac{\sqrt{t-s}e^{-(t-s)-\frac{(-\sqrt{2}\gamma t-m_s+K)^2}{2(t-s)}}}{-2s+\sqrt{2}\gamma\eta(t-s)}\dd s\cr&\qquad+c_\delta^2\int_{\varepsilon t}^t e^{-(t-s)(1-\gamma^2(1+\eta)^2)-\sqrt{2}\gamma(1+\eta)(m_s-\sqrt{2}\gamma t)}\dd s\cr
&=:U_{\eqref{U2F}b1}+U_{\eqref{U2F}b2},
\end{align*}
that we bound separately.

Note that
\begin{align*}
U_{\eqref{U2F}b2}
&=\int_\varepsilon^1 (ut)^{3\gamma(1+\eta/2)}t e^{-t[(1-u)(1-\gamma^2(1+\eta)^2)+2\gamma(1+\eta)(u+\gamma)]}\dd u\cr
&\leq  t^{\frac{3\gamma}{2}(1+\eta)+1} e^{-t\min_{u\in[\varepsilon,1]}[(1-u)(1-\gamma^2(1+\eta)^2)+2\gamma(1+\eta)(u+\gamma)]}.
\end{align*}
One can check that
\begin{align*}
&\min_{u\in[\varepsilon,1]}[(1-u)(1-\gamma^2(1+\eta)^2)+2\gamma(1+\eta)(u+\gamma)]
\cr&\quad=
(1+\gamma^2)+\varepsilon(1+\gamma^2)\eta-\eta^2\gamma^2(1-\varepsilon).
\end{align*}
Take $\varepsilon\in (\frac{\eta\gamma^2}{1+\gamma^2},\frac{\gamma \eta}{1+2\gamma(1-\delta)})$ as $\eta=1-2\delta\in(0,1)$. Then,
\[
U_{\eqref{U2F}b2}\leq t^{\frac{3\gamma}{2}(1+\eta)+1} e^{-t(1+\gamma^2+\varepsilon\eta^2\gamma^2)}=o_t(1)t^{3\gamma/4}e^{-t(1+\gamma^2)}.
\]
It remains to bound $U_{\eqref{U2F}b1}$. Recalling \eqref{galpha}, we observe that
\begin{align*}
U_{\eqref{U2F}b1}
\leq &\frac{C_{\delta,\epsilon}^{(7)}}{\sqrt{t}}\int_{A\sqrt{t}}^{\epsilon t} e^{-(t-s)-\frac{(-\sqrt{2}\gamma t-m_s+K)^2}{2(t-s)}}\dd s\\
\leq & C_{\delta,\epsilon}^{(7)}e^{C^{(8)}_{\delta,\epsilon}K}\sqrt{t}\int_{\frac{A}{\sqrt{t}}}^\epsilon e^{-tg_{-\gamma}(u)+\frac{3}{2}\log(ut)\frac{u+\gamma }{1-u}}\dd u,
\end{align*}
where we use the fact that for $s\in [{A\sqrt{t}},\, {\epsilon t}], $
\begin{align}\label{expanation}\frac{(-\sqrt{2}\gamma t-m_s+K)^2}{2(t-s)}&=\frac{2(\gamma t+s)^2+(\frac{3}{2\sqrt{2}}\log s+K)^2-2\sqrt{2}(\gamma t+s)(\frac{3}{2\sqrt{2}}\log s+K)}{2(t-s)}\cr
&\geq \frac{(\gamma t+s)^2}{t-s}-\frac{3(\gamma t+s)\log s}{2(t-s)}-\frac{\sqrt{2}\gamma K}{(1-\epsilon)}. \end{align}
Since
$g_{-\gamma}(u)=1+\gamma^2+2u^2+o(u^2),$ as $u\downarrow0$, then
\[
 \sqrt{t}\int_{\frac{\log t}{\sqrt{t}}}^\epsilon e^{-tg_{-\gamma}(u)+\frac{3}{2}\log(ut)\frac{u+\gamma }{1-u}}\dd u \leq \sqrt{t}\int_{\frac{\log t}{\sqrt{t}}}^\epsilon (ut)^{3(\epsilon+\gamma)/2}e^{-u^2t-(1+\gamma^2)t}\dd u,
\]
which is $o_t(1)t^{3\gamma/4}e^{-(1+\gamma^2)t}$. For $u\in[\frac{A}{\sqrt{t}},\frac{\log t}{\sqrt{t}}]$, $\log (ut)\frac{u+\gamma}{1-u}= \gamma\log (ut)+o_t(1)$. Therefore,
\begin{multline*}
\sqrt{t}\int_{\frac{A}{\sqrt{t}}}^{\frac{\log t}{\sqrt{t}}} e^{-tg_{-\gamma}(u)+\frac{3}{2}\log(ut)\frac{u+\gamma }{1-u}}\dd u
\leq  e\int_{\frac{A}{\sqrt{t}}}^{\frac{\log t}{\sqrt{t}}} (ut)^{3\gamma/2}e^{-t(1+\gamma^2)-u^2t}\dd u\\
  \leq et^{3\gamma/4}e^{-(1+\gamma^2)t}\int_{A}^{\log t} x^{3\gamma/2}e^{-2x^2}\dd x = o_A(1)t^{3\gamma/4}e^{-(1+\gamma^2)t}.
\end{multline*}
We have completed the proof of \eqref{largebadtimec}.

\textbf{Proof of \eqref{smallbadtimec}:} We have
\begin{align}\label{U2G}
&U_2(-\sqrt{2}\gamma t, t, [0,{\sqrt{t}}/{A}])\cr&\quad=U_2(-\sqrt{2}\gamma t, t, [0,{\sqrt{t}}/{A}],[-K,\infty))+U_2(-\sqrt{2}\gamma t, t, [0,{\sqrt{t}}/{A}],[-\infty, -K])
\cr&\quad=:U_{\eqref{U2G}a}+U_{\eqref{U2G}b}.
\end{align}
As $u(m_s+a,s)\leq 1$, applying \eqref{basic} gives that for $t$ large enough,
\begin{align*}
U_{\eqref{U2G}a}&\leq \int_0^{\sqrt{t}/A} \frac{\sqrt{t-s}}{-K+m_s+\sqrt{2}\gamma t}e^{-(t-s)-\frac{(m_s+\sqrt{2}\gamma t-K)^2}{2(t-s)}}\dd s\cr
&\leq c_7e^{\sqrt{2}K}e^{-(1+\gamma^2)t}\int_0^{\sqrt{t}/A} s^{3\gamma/2}e^{-\frac{2s^2}{t}}\frac{\dd s}{\sqrt{t}}
\cr&=c_7e^{\sqrt{2}K}t^{3\gamma/4}e^{-(1+\gamma^2)t}\int_0^{1/A}u^{3\gamma/2}e^{-2u^2}\dd u,
\end{align*}
which is $o_A(1)t^{3\gamma/4}e^{-(1+\gamma^2)t}$. Similarly as $U_{\eqref{U2F}b1}$,
\begin{align*}
U_{\eqref{U2G}b}\leq &C^{(12)}_{\delta,\epsilon} e^{C^{(11)}_{\delta,\epsilon}K}\sqrt{t}\int_0^{\frac{1}{A\sqrt{t}}} e^{-tg_{-\gamma}(u)+\frac{3\gamma}{2}\log(ut)\frac{u+\gamma }{1-u}}\dd u\\
\leq & C^{(1)}_{\delta,\epsilon,K} t^{3\gamma/4}e^{-(1+\gamma^2)t}\int_0^{\frac{1}{A\sqrt{t}}} (u\sqrt{t})^{3\gamma/2}e^{-u^2t}\sqrt{t}\dd u\\
=& c_8 t^{3\gamma/4}e^{-(1+\gamma^2)t}\int_0^{1/A}u^{3\gamma/2}e^{-u^2}\dd u=o_A(1)t^{3\gamma/4}e^{-(1+\gamma^2)t},
\end{align*}
concluding \eqref{smallbadtimec}.
\end{proof}

\begin{proof}[Proof of Lemma \ref{badpositionc}]
\textbf{Proof of \eqref{badsmallpositionc}:}
Take $\delta\in(0,1/3)$ and $\eta=1-2\delta$. By similar reasoning as above, we have
\begin{align*}
&\int_{\frac{\sqrt{t}}{A}}^{A\sqrt{t}}\dd s\int_{-\infty}^{-K}\dd z \frac{e^{-(t-s)-\frac{(z+m_s-\sqrt{2}ct)^2}{2(t-s)}}}{\sqrt{2\pi (t-s)}}u^2(m_s+z,s)\\\leq &c_\delta^2\int_{\frac{\sqrt{t}}{A}}^{A\sqrt{t}} \frac{\sqrt{t-s}e^{-(t-s)-\frac{(-\sqrt{2}\gamma t-m_s+K)^2}{2(t-s)}-\sqrt{2}\gamma(1+\eta)K}}{-2s+\sqrt{2}\gamma\eta(t-s)}\dd s\\
\leq & C^{(1)}_{\delta, \gamma, A}t^{3\gamma/4}e^{-K\sqrt{2}\gamma(1-3\delta)} \int_{\frac{\sqrt{t}}{A}}^{A\sqrt{t}}\frac{1}{\sqrt{t}}e^{-(t-s)-\frac{(s+\gamma t)^2}{t-s}}\dd s
\cr\leq& C^{(2)}_{\delta, \gamma, A} e^{-K\sqrt{2}\gamma(1-2\delta)}t^{3\gamma/4} e^{-(1+\gamma^2)t},
\end{align*}
where   for the second inequality, we used the fact that for $s\in [{\frac{\sqrt{t}}{A}},\, {A\sqrt{t}}], $
\begin{align*}\frac{(-\sqrt{2}\gamma t-m_s+K)^2}{2(t-s)}
&\geq \frac{(\gamma t+s)^2}{t-s}-\frac{3(\gamma t+s)\log s}{2(t-s)}-\frac{\sqrt{2}(\gamma t+s)\gamma K}{(t-s)}
\cr&\geq \frac{(\gamma t+s)^2}{t-s}-\frac{3\gamma}{4}\log t-\sqrt{2}\gamma^2K+o_t(1).
\end{align*}
and the last inequality follows from the fact that $(t-s)+\frac{(s+\gamma t)^2}{t-s}=(1+\gamma^2)t+\frac{(1+\gamma)^2s^2}{t-s}$.

\textbf{Proof of \eqref{badlargepositionc}:}
 For $z\geq K$, using the fact $u(m_s+z,s)\leq 1$, we obtain that
\begin{align*}
U_2(-\sqrt{2}\gamma t, t, [\sqrt{t}/{A}, A\sqrt{t}], [K,\infty))\leq & \int_{\frac{\sqrt{t}}{A}}^{A\sqrt{t}}\dd s\int_K^{\infty}\dd z \frac{e^{-(t-s)-\frac{(z+m_s+\sqrt{2}\gamma t)^2}{2(t-s)}}}{\sqrt{2\pi (t-s)}}\\
=&\int_{\frac{\sqrt{t}}{A}}^{A\sqrt{t}}e^{-(t-s)}\dd s\int_{K+m_s+\sqrt{2}\gamma t}^{\infty}\dd z \frac{e^{-\frac{z^2}{2(t-s)}}}{\sqrt{2\pi (t-s)}},
\end{align*}
which by \eqref{basic} is less than
\begin{align*}
\int_{\frac{\sqrt{t}}{A}}^{A\sqrt{t}}e^{-(t-s)}\tfrac{\sqrt{t-s}}{K+m_s+\sqrt{2}\gamma t}e^{-\frac{(K+m_s+\sqrt{2}\gamma t)^2}{2(t-s)}}\dd s
\leq \frac{C_{\gamma, A}}{\sqrt{t}} \int_{\frac{\sqrt{t}}{A}}^{A\sqrt{t}}e^{-(t-s)-\frac{(m_s+\sqrt{2}\gamma t)^2}{2(t-s)}-K\frac{\sqrt{2}\gamma t+m_s}{t-s}}\dd s.
\end{align*}
Similarly as above, we end up with
\[
U_2(-\sqrt{2}\gamma t, t, [\sqrt{t}/{A}, A\sqrt{t}], [K,\infty))\leq C_{\gamma, A} e^{-K\sqrt{2}\gamma} t^{3\gamma/4}e^{-(1+\gamma^2)t}.
\]
This suffices to conclude \eqref{badlargepositionc}.
\end{proof}

\noindent{\textsc{Xinxin Chen}}

\noindent{Institut Camille Jordan, C.N.R.S. UMR 5208, Universite Claude Bernard Lyon 1, 69622 Villeurbanne Cedex, France.}

\noindent{E-mail: {\tt xchen@math.univ-lyon1.fr}}

\bigskip

\noindent{\textsc{Hui He}}

\noindent{School of Mathematical Sciences, Beijing Normal University,
Beijing 100875, People's Republic of China.}

\noindent{E-mail: {\tt hehui@bnu.edu.cn}}

\bigskip 

\noindent{\textsc{Bastien Mallein}}

\noindent{Université Sorbonne Paris Nord, LAGA, UMR 7539, F-93430, Villetaneuse, France.}

\noindent{E-mail: {\tt mallein@math.univ-paris13.fr}}

\end{document}